\documentclass[12pt]{amsart}

\usepackage{xcolor}

\usepackage{hyperref}
\usepackage{bbm}

\usepackage{graphicx}
\usepackage{eucal,amsmath,amssymb}
\usepackage{amsthm}
\usepackage{ulem}

\usepackage{amsmath}%
\usepackage{amsfonts}%
\usepackage{amssymb}%
\usepackage{graphicx}

\usepackage{t1enc,epsfig}

\newtheorem{theorem}{Theorem}
\theoremstyle{plain}

\newtheorem{definition}{Definition}

\newtheorem{proposition}{Proposition}
\newtheorem{remark}{Remark}

\numberwithin{equation}{section}

\newcommand{\srs}{\mathcal{S}_{(r,s)}}

\def\N{\mathbb{N}}   
\def\R{\mathbb{R}}   
\def\C{\mathbb{C}}   
\def\T{\mathbb{T}}   

\newcommand{\kn}{k=1, \dots, n}

\newcommand{\norm}[1]{\|#1\|}

\newcommand{\vb}{\mathcal V}
\newcommand{\wb}{\mathcal W}
\newcommand{\vu}{\underline v}
\newcommand{\wu}{\underline w}

\newcommand{\pb}{\mathcal P}
\newcommand{\mnh}{M^h_n(\C)}

\newcommand{\frs}{\mathcal{F}_{(r,s)}}

\DeclareMathOperator{\tr}{Tr}

\DeclareMathOperator{\gr}{Gr}
\DeclareMathOperator{\rank}{rank}

\DeclareMathOperator{\grad}{grad}
\DeclareMathOperator{\im}{Im}

\begin{document}
%
%
%
%
%
%
%
%
%
%
	
	\title[Supports for Minimal Hermitian Matrices]{Supports for Minimal Hermitian Matrices}

	\author[A. Mendoza]{Alberto Mendoza}
	\address{Universidad Sim\'{o}n Bol\'{\i}var\\ Apartado 89000\\ Caracas 1080A\\ Venezuela
	}
	\email{jacob@usb.ve}

	\author[L. Recht]{L\'azaro Recht}
	\address{Universidad Sim\'{o}n Bol\'{\i}var\\ Apartado 89000\\ Caracas 1080A\\ Venezuela	
		\\and\\ Instituto Argentino
		de Matem\'atica ``Alberto P. Calder\'on''\\ Saavedra 15 3$^\text{er}$ piso\\
		(C1083ACA) Ciudad Aut\'onoma de Buenos Aires, Argentina
	}
	\email{recht@usb.ve}

	\author[A. Varela]{Alejandro Varela}
	\address{Instituto de Ciencias, Universidad Nacional de General Sarmiento\\
		J.M. Gutierrez 1150\\ 
		(B1613GSX) Los Polvorines, Pcia. de Buenos Aires, Argentina
		\\and\\ Instituto Argentino
		de Matem\'atica ``Alberto P. Calder\'on'', Saavedra 15 3$^\text{er}$ piso,
		(C1083ACA) Ciudad Aut\'onoma de Buenos Aires, Argentina} 
		\email{avarela@ungs.edu.ar}

	
	\begin{abstract}
		We study certain pairs of subspaces $V$ and $W$ of $\C^n$ we call supports that consist of eigenspaces of the eigenvalues $\pm\|M\|$ of a minimal hermitian matrix $M$ ($\|M\|\leq \|M+D\|$ for all real diagonals $D$).
		
		For any pair of orthogonal subspaces we define a non negative invariant $\delta$ called the adequacy to measure how close they are to form a support and to detect one. This function $\delta$ is the minimum of another map $F$ defined in a product of spheres of hermitian matrices. We study the gradient, Hessian and critical points of $F$ in order to approximate $\delta$. These results allow us to prove that the set of supports has interior points in the space of flag manifolds. 
		\end{abstract}
	
	\date{\today}
	\maketitle

%

\section{Introduction}
Let $\C^n$ denote as usual the vector space of $n$-tuples of complex numbers 
and $M_n(\C)$ the $n\times n$ complex matrices. 
Let $M^h_n(\C)$ (respectively $M^{ah}_n(\C)$) be the set of hermitian or self-adjoint (respectively anti-hermitian) matrices and $\|\ \|$ the spectral norm in $M_n(\C)$.  

We call $Z\in\mnh$ a minimal 
matrix if 
\begin{equation}
\label{eq: def matriz minimal}
\|Z\|\leq \|Z+D\|\ , \text{ for every real diagonal  } \ D,
\end{equation}
(a similar definition can be given for antihermitian matrices and pure 
imaginary  
 diagonals). 
Minimal 
matrices allow the description of short length curves in the homogeneous 
space $\mathcal{P}= U(n )/ U( \mathcal{D}_n)$, where
$\mathcal{D}_n$ denotes the diagonal $n\times n$ complex matrices, 
$U(\mathcal{D}_n)$ the unitary diagonal matrices and $U(n)$ the unitary group of $ M_n(\C)$. 
%
More precisely, we consider the homogeneous space $\mathcal{P}$, with the 
left action $L_U(\rho)=U\rho U^*$, for $U\in U( n)$, $\rho\in\mathcal{P}$ (where the action is performed on any element of the class $\rho$). Then the space $\mathcal{P}$ is provided with the invariant Finsler metric defined by the quotient norm in $M_n(\C)^{ah}/\mathcal{D}_n^{ah}$, the tangent space of $\mathcal{P}$ at $\rho$. This structure allows the definition of a natural distance $\text{d}(\rho_1,\rho_2)$ in $\mathcal{P}$ as the infimum of the length of curves in $\mathcal{P}$ joining $\rho_1$ and $\rho_2$ (see \cite{bottazzi_varela_LAA, dmr1} for details).

The following result is a restatement of Theorem I of \cite{dmr1} in the 
present context.

\begin{theorem}\label{teo: Theorem I de dmr1} 
	Let $\rho\in\mathcal{P}=U(n )/ U( \mathcal{D}_n)$, $X\in(T\mathcal{P})_{\rho}\simeq M_n(\C)^{ah} 
	/\mathcal{D}_n^{ah}$ and $Z\in 
	M_n^{ah}(\C)$ a minimal matrix which projects to $X$ ($X=Z\rho-\rho Z$).
	Then the curve given by $\gamma(t)=L_{e^{tZ}}\rho={e^{tZ}}\rho {e^{-tZ}}$ satisfies 
	$\gamma(0)=\rho$, $\dot \gamma(0)=X$ and
	has minimal length among all curves in $\mathcal{P}$ joining 
	$\gamma(0)$ to $\gamma(t)$ for each $t$ with $|t|\leq\frac{\pi}{2||Z||}$.
\end{theorem} 

Observe that from all the cases covered by Theorem I of \cite{dmr1}, the homogeneous space $\mathcal{P}=U(n)/U( \mathcal{D}_n)$ we are considering here is probably the simplest non commutative non trivial case.
 
This result motivates the study of minimal matrices in the spectral norm. Some particular properties have been studied already but in the present work we focus on the particular and rich structure of a spectral pair of eigenspaces related to a minimal matrix.

If $Z\in M_n^h(\C)$  is a minimal matrix, then $\pm\|Z\|$ must be eigenvalues of $Z$.
Nevertheless, this condition is not enough. 
If $\pm\|Z\|\in \sigma(Z)$, then $Z$ is minimal if and only if there exist orthogonal corresponding eigenspaces $V_+$ and $V_-$ (ranges of the spectral projections $P_{V_+}$ and $P_{V_-}$ corresponding to the 
eigenvalues $\pm \|Z\|$) that satisfy
\begin{equation}\label{eq: estructura matriz minimal}
Z= \|Z\| P_{V_+}-\|Z\| P_{V_-}+R
\end{equation}
(where $R\in M_n^h(\C)$, its range is orthogonal to $V_+\oplus V_-$, and 
$\|R\|\leq\|Z\|$) and such that $V_+$ and $V_-$ satisfy the following property:

\textbf{Condition 1}. There exist orthonormal sets $\{v_i\}_{i=1}^{p}\subset V_+$ 
and  
$\{w_j\}_{j=q+1}^{p+q}\subset V_-$ such that
\begin{equation}\label{interseccion convexa no vacia BON particular}
\text{co}\big( \{v_i\circ\overline{v_i}\}_{i=1}^{r} \big) \cap
\text{co}\left(\{w_j\circ\overline{w_j}\}_{j=r+1}^{r+s}\right) \neq \emptyset
\end{equation}
where $\circ$ denotes the Hadamard or entrywise product and $\text{co}(A)$ the 
convex hull of $A$ (see Corollary 3 in 
\cite{andruchow_larotonda_recht_varela_Companneras}).

In Theorem \ref{teo: equival def soporte} it is proved that Condition 1 is equivalent to the 
following property held by two orthogonal subspaces $V$ and $W$.

\begin{definition}\label{def: soporte}
Given two orthogonal subspaces $V$ and $W\subset \C^n$ we call the pair $(V,W)$ a {\bf support} if there exist non trivial subsets $\{v^1,v^2,\dotsc,v^p\}$ of $V$ and $\{w^1,w^2,\dotsc,w^q\}$ of $W$ with coordinates in the canonical basis given by $v^i=(v^i_1,v^i_2,\dotsc,v^i_n)$, for $i=1,\dotsc,p$ and $w^j=(w^j_1,w^j_2,\dotsc,w^j_n)$, for $j=1,\dotsc,q$ such that 
\begin{equation}
\label{eq: condicion soporte} 
\left\{
\begin{array}{ccc}
|v_1^1|^2+|v_1^2|^2+\dotsb+|v_1^p|^2&=&|w_1^1|^2+|w_1^2|^2+\dotsb+|w_1^q|^2\\
|v_2^1|^2+|v_2^2|^2+\dotsb+|v_2^p|^2&=&|w_2^1|^2+|w_2^2|^2+\dotsb+|w_2^q|^2\\
\vdots&\vdots&\vdots\\
|v_n^1|^2+|v_n^2|^2+\dotsb+|v_n^p|^2&=&|w_n^1|^2+|w_n^2|^2+\dotsb+|w_n^q|^2
\end{array}
\right.
\end{equation}
or equivalently $\sum_{i=1}^p v^i\circ\  \overline{v^i}=\sum_{j=1}^q w^j\circ\ 
\overline{w^j}$, where $\circ$ denotes the Hadamard product and $\overline{x}$ 
the vector formed by the conjugated coordinates of $x$.

This definition can also be stated choosing orthogonal vectors $\{v^i\}_{i=1}^t$ and  $\{w^j\}_{j=1}^h$ (see Theorem \ref{teo: en def. soporte se puede tomar l.i. y ortog.}).

\end{definition}

\begin{remark}
		The previous discussion implies that $Z$ is a minimal matrix with a decomposition as in \eqref{eq: estructura matriz minimal} if and only if the pair of subspaces $(V_+,V_-)$ is a support (see also Theorem \ref{teo: equival def soporte}).
	\end{remark}

 We will denote with $\mathcal{S}_{(r,s)}$ the set of supports of $\C^n$ with corresponding dimensions $r$ and $s$:
	\begin{equation}
	\label{notacion conjunto soportes}
	\begin{split}
	\mathcal{S}_{(r,s)}=&\left\{(V,W)\in \C^n\times \C^n: (V,W) \text{ is a support}\right.\\
	& \qquad\ \ \left. \text{ with } \dim(V)=r \text{ and } \dim(W)=s\right\}
	\end{split}
	\end{equation}
 
	\begin{remark}\label{rem: variedad algebraica}
		The definition of the set of supports suggests that it might have the structure of a real algebraic set. As expected, $\srs$ turns out to be closed (see Proposition \ref{prop: Srs es cerrado en las banderas}). Nevertheless, the fact that for every $n\in\N_{n\geq 3}$ there exist interior points in $\srs$ in the ambient of a flag manifold of $\C^n$ is a surprising result (see Section \ref{sec: soportes con entornos de soportes}). It would be interesting to find out if $\srs$ is a semi-algebraic set.
	\end{remark}
 
The previous comments allow us to state the following result.

\begin{remark}
	\label{teo: funcion entre minimales y soportes}
	There exists a function between the set of minimal matrices $Z$ with eigenspaces $V_+$ and $V_-$ corresponding to the eigenvalues $\pm \|Z\|$ with $\dim(V_+)=r$, $\dim(V_-)=s$ onto $\srs$ that maps $Z$ to the support $(V_+,V_-)$. 

Consider the equivalence class of a matrix $M\in M_n^h(\C)$ defined by $[M]=\{ N\in M_n^h(\C): (M-N) \in \mathcal{D}_n\}$. The relation between $[M]$ and the support determined by its corresponding minimal matrix (or matrices) is a work in progress that will studied elsewhere. 
	\end{remark}

Supports are a fundamental aspect of the description of minimal matrices.
In this work we are going to analyze the structure of the set of supports $\srs$ as a subset of the flag manifold $\mathcal{F}_{(r,s)}$ (see \eqref{def: F(r,s)}) under the identification of $(V,W)\in\srs$ with $V\oplus W\oplus  (\C^n\ominus V\ominus W)$.
The authors consider that the study of $\srs\subset \mathcal{F}_{(r,s)}$ is interesting by itself.

In order to measure how far are two subspaces $V$ and $W$ to become a support we define in \eqref{def: definicion de adequacy} a number $\delta(V,W)\geq 0$ we call the adequacy of $V$ and $W$ that satisfies $\delta(V,W)=0$ if and only if $(V,W)$ is a support. The adequacy is a natural tool to achieve this and can be computed as the minimum of a function $F$ defined on the product of certain spheres $S_V\times S_W$ of linear maps (see \eqref{def: S V y S W} and \eqref{def: funcion F}). We study the gradient, Hessian and critical points of this $F$ (see Section \ref{sec: adecuacion}) to allow the approximation of the adequacy. Some of the formulas obtained are used in the appendices to obtain numerical examples of particular supports that are interior points of flag manifolds in low dimensions. These results are used to prove in Theorem \ref{teo: los soportes forman un abierto} that there exist open neighborhoods of supports (formed by supports in $\C^n$ for every $n\in\N_{n\geq 3}$) in the flag manifold $\mathcal{F}_{(\dim(V),\dim(W))}$.

 We also consider a geometric interpretation of the adequacy in Section \ref{sec: interpret geom de adecuacion} describing a new space of parameters to calculate it. This perspective allows the characterization of some critical points of the map $F$ whose global minimum is the adequacy in sections \ref{sec: On the critical points of the function F} and \ref{sec: Characterization of critical points of F}. 

\section{Preliminaries and notation}\label{sec: preliminares}
Here we introduce some notation used throughout the article.
$M_n(\C)$ will denote the $n\times n$ matrices with coefficients in $\C$, 
$M_n^h(\C)$ the hermitian matrices and $M_n^{ah}(\C)$ the anti-hermitian 
matrices.
The expression diag$(a_1,a_2,\dots,a_n)$ denotes the diagonal matrix in 
$M_n(\C)$ with the elements $a_1,a_2,\dots,a_n\in \C$ in its principal diagonal, and $\Phi:M_n(\C)\to \mathcal{D}_n\subset M_n(\C)$ the conditional expectation such that $\Phi(x)$ is the diagonal matrix formed with the diagonal entries of $x$.

As usual $GL(n,\C)$ denotes the general group of invertible matrices in 
$C^{n\times n}$. And $\gr(k,\C^n)$ will denote the Grassmannian manifold of all 
$k$-dimensional subspaces of $\C^n$.


We denote with $v\circ w$ the Hadamard (or Schur) entrywise product of two 
vectors 
$v$, 
$w\in\C^n$, where $v\circ w\in\C^n$, and $(v\circ w)_i=v_i w_i$, for $i=1,\dots 
n$. Similarly $A\circ B$ is the Hadamard (or Schur) product of two matrices $A, B\in M_n(\C)$.

We use $\mathcal{F}_{(r,s)}$ to represent the set
\begin{equation}
\label{def: F(r,s)}
\begin{split}
\mathcal{F}_{(r,s)}=\{(V,W):  V\perp W & \text{ are subspaces of }\C^n,\\
& \dim(V)=r, \dim(W)=s\}.
\end{split}
\end{equation}
Observe that the pair $(V,W)\in\mathcal{F}_{(r,s)}$ can be identified with the element $\{0\}\subset V\subset V\oplus W\subset \C^n$ in a classic flag manifold $F(r,r+s,n)$ which is isomorphic to the homogeneous space $U(n)/\big(U(r)\times U(s)\times U(n-r-s)\big)$. Therefore, $\mathcal{F}_{(r,s)}$ can be identified with the flag manifold $F(r,r+s,n)$.

\section{Properties of a support $(V,W)$}\label{seccion: The basis in the definition of a support can be 
taken orthogonal}


Given $v^1,v^2,\dotsc,v^p\in V$, for $V$ a subspace of $\C^n$, we denote with
 \begin{equation}\label{def: v barra}
 \underline{v}=(v^1,v^2,\dots, v^p)=
\left(
\begin{array}{cccc}
 v_1^1 & v_1^2 &\dots &v_1^p\\
 \vdots &\vdots&\dots &\\
v_n^1 & v_n^2 &\dots &v_n^p
\end{array}
\right)
\end{equation}
either the $n\times p$ matrix or the $\C^p\to \C^n$ linear map.
Let $\Phi:M_n(\C)\to \mathcal{D}_n$ be the map such that $\Phi(x)$ is the diagonal of $x$. Then the $n$-tuples that appear in \eqref{eq: condicion soporte}  can be written
\begin{equation}\label{eq: def diag vu.vu}
\sum_{i=1}^p v^i\circ \overline{v^i} = \left(
 |v_1^1|^2 + \dots +|v_1^p|^2, \  \dots \ ,
  |v_n^1|^2 + \dots +|v_n^p|^2\right)\simeq \Phi(\underline{v}\ \underline{v}^*)
\end{equation}
where we identified the vector with the diagonal matrix $\Phi(\underline{v}\ \underline{v}^*)$ of $\underline{v}\ \underline{v}^*\in M^h_n(\C)$. 

Then using the singular value decomposition of $\vu=u s x^*$, with $u\in U(n)$, $x\in U(p)$ and $s$ the $n\times p$ diagonal matrix of the singular values of $\vu$ in the $s_{j,j}$ entries. Let us denote them with $s_j$, $j=1,\dots, p$. 
Now consider the column vectors $u^i\in \C^n$, $i=1,\dots, n$, from the unitary matrix $u$. Note that these $u^i$ are eigenvectors of $\vu\ \vu^*$.

Let  $e_i$, for $i=1,\dots, n$, be the $i^{\text{th}}$ element of the canonical basis of $\C^n$. Then the $i^{\text{th}}$ diagonal elements of $\Phi(\vu\ \vu^*)$ are
\begin{equation}\label{eq: diagonal de vu igual a diagonal de autovectores por valores singulares}
\langle \vu\ \vu^* e_i, e_i\rangle=\langle x s^t u^* e_i, x s^t u^* e_i\rangle=\sum_{j=1}^p s_j^2|u^j_i|^2=\sum_{j=1}^p \langle s_j u^j_i,s_j u^j_i\rangle.
\end{equation}

Therefore if we consider the $n\times p$ matrix given by its columns
\begin{equation}
\label{eq: base ortogonal con igual momento que vectores generadores}
\widetilde{\underline{v}}=(s_{1} u^{1} , s_{2} u^{2},\dots ,s_{p} u^{p})
\end{equation}
then the computation made in (\ref{eq: diagonal de vu igual a diagonal de autovectores por valores singulares}) proves that 
 \begin{equation}\label{eq: momentos entre vectores y vectores ortogonales coincide}
 \Phi(\vu\ \vu^*)=\Phi(\widetilde{\underline{v}}\ \widetilde{\underline{v}}^*).
 \end{equation} 
Moreover, these columns generate the same subspace than the original $\{v^j\}_{j=1}^p$.

Let $K\subset \{1,\dots, p\}$ be the subset of indexes such that $s_k\neq 0$ if and only if $k\in K$ and let $t=\#(K)$. Then the vectors $\{s_k u^k\}_{k\in K}\subset \C^n$ are orthogonal to each other and generate the same subspace than the original columns $v^j$ of $\vu$ for $j=1,\dots, p$.

Therefore, if we consider the $n\times t$ matrix with columns $s_k u^k$, $k\in K$
\begin{equation}
\label{eq: base de vectores ortogonal con igual momento}\widetilde{\underline{v'}}=
(s_{k_1} u^{k_1} , s_{k_2} u^{k_2},\dots ,s_{k_t} u^{k_t})
\end{equation}
then its columns form an orthogonal basis of the subspace generated by $\{v^j\}_{j=1}^r$ and it is apparent that also $\Phi\left(\widetilde{\underline{v'}}\ \widetilde{\underline{v'}}^*\right)=\Phi(\widetilde{\underline{v}}\ \widetilde{\underline{v}}^*)=\Phi(\vu\ \vu^*)$.

Therefore, we have proved the following result.

\begin{theorem}\label{teo: en def. soporte se puede tomar l.i. y ortog.}
If $(V,W)$ is a support in $\C^n$ as in Definition \ref{def: soporte} then there exists not null orthogonal vectors $\{v^i\}_{i=1}^t\subset V$ and $\{w^j\}_{j=1}^h\subset W$ that satisfy equation \eqref{eq: condicion soporte}, or equivalently $\sum_{i=1}^t v^i\circ\  \overline{v^i}=\sum_{j=1}^h w^j\circ\ 
\overline{w^j}$.
\end{theorem}

\begin{remark}\label{rem: vectores de condicion soporte pueden ser acotados}
 		Observe that in Definition \ref{def: soporte} the vectors $\{v^i\}_{i=1,\dots ,r}$ of the subspace $V$ are not required to be linearly independent nor generators of $V$, and similarly for $\{w^j\}_{j=1,\dots ,s}$ in $W$, but the previous theorem states that orthonormal vectors can be chosen. Moreover, these vectors can be taken bounded in norm with a fixed constant $C$ after multiplying all of them by $\frac C{\|x\|}$ where $\|x\|$ is the greatest norm of all the vectors considered.
\end{remark}

\begin{definition}\label{def: momento del sistema}
If $\{v^1, \dotsc,v^p\}$ is a system of $p$ vectors in $\C^n$, then the diagonal matrix (or corresponding vector) $\Phi(\underline{v}\ \underline{v}^*)$ will be called the \textbf{moment} of the system $\{v^1,\dotsc,v^p\}$ (with the notation of $\underline{v}\in \C^{n\times p}$ detailed in \eqref{def: v barra}).
\end{definition}

Therefore the previous discussion also proved the following result.
\begin{proposition}
 \label{teo: base ortogonal con mismo momento}
 If $\{v^1,\dotsc,v^p\}$ is a system of $p$ linearly independent vectors in $V$, then there is an orthogonal basis $\{c^1,\dotsc, c^p\}$ of $V$ with the same moment than that of $\{v^1,\dotsc,v^p\}$.
 \end{proposition}

\begin{remark} \label{rem: equiv soporte igualdad diags} Observe that if for $V\perp W$ subspaces of $\C^n$, there exist $\{v^1, \dots,v^p\}\subset V$, $\{w^1, \dots,w^q\}\subset W$, and we define $\vu\in \C^{n\times p}$ and $\wu\in\C^{n\times q}$ as in \eqref{def: v barra}, then the equality 	
	\begin{equation}
	\Phi(\vu\ \vu^*)=\Phi(\wu\ \wu^*) 
	\end{equation}
	is equivalent to the fact that $(V,W)$ is a support in $\C^n$ (see \eqref{eq: def diag vu.vu}).
\end{remark}

Given a support $(V,W)$ of $\C^n$ Proposition \ref{teo: base ortogonal con mismo momento} implies that there exists an orthogonal set $\left\{{v^i}\right\}_{i=1}^{p}$ for $V$ and $\left\{{w^j}\right\}_{j=1}^{q}$ for $W$ that satisfy \eqref{eq: condicion soporte}. Now consider the orthonormal corresponding set after normalizing each vector.
Now adding all the equations in (\ref{eq: condicion soporte}) we obtain that $\sum_{i=1}^p\|v^i\|^2=\sum_{j=1}^q\|w^j\|^2$, and then
$$
\sum_{i=1}^p \frac{\|v^i\|^2}{\sum_{k=1}^p \|v^k\|^2}
\left(\frac{v^i}{\|v^i\|}\circ  {\frac{\overline{v^i}}{\|v^i\|}}\right) = 
\sum_{j=1}^q \frac{\|w^j\|^2}{\sum_{k=1}^q \|w^k\|^2}
\left(\frac{w^j}{\|w^j\|}\circ {\frac{\overline{w^j}}{\|w^j\|}}\right)
$$

which in turns implies that Condition 1 stated in \eqref{interseccion convexa no vacia BON particular} holds.
Then statement (b) of Corollary 3 in 
\cite{andruchow_larotonda_recht_varela_Companneras} is fulfilled and $M=\lambda\ P_V-\lambda\ P_W+R\in M^h_n(\C)$ (with $P_VR=P_W 
R=0$, $R\in M_n^h(\C)$, $\|R\|\leq\lambda >0$) is a minimal matrix in the sense 
that $\|M\|\leq \|M+D\|$ for all real diagonal matrices $D\in \mnh$ and $\|\ 
\|$ the spectral norm (see 
\cite{andruchow_larotonda_recht_varela_Companneras}). Then a support allows the 
construction of a minimal matrix, and vice versa. 
In the following theorem we collect some statements that are equivalent to the definition of a support.

\begin{theorem}\label{teo: equival def soporte}
	Let $V, W$ be two non trivial orthogonal subspaces of $\C^n$, then the following statements are equivalent.
	 \begin{enumerate}
	 	\item $(V,W)$ is a support, that is, there exist non trivial subsets $\{v^1,v^2,\dotsc,v^p\}$ of $V$ and $\{w^1,w^2,\dotsc,w^q\}$ of $W$ such that \eqref{eq: condicion soporte} holds.
	 	\item The hermitian matrix $M=\lambda (P_V-P_W)+R$ is minimal (see \eqref{eq: def matriz minimal}) for every $\lambda\in\R_{>0}$, $R\in M_n^h(\C)$, $\|R\|\leq \lambda$, $R(P_V +P_W)=0$.
	 	\item There exist non trivial subsets $\{v^1,v^2,\dotsc,v^p\}$ of $V$ and $\{w^1,w^2,\dotsc,w^q\}$ of $W$ such that 
	 	$$ \Phi(\vu \ \vu^*)=\Phi(\wu \ \wu^*)$$
	 	with $\vu$ and $\wu$ defined as in \eqref{def: v barra} and $\Phi(m)$ the diagonal of $m$. 
	 	\item The sets 	
	 	$
	 	\sigma_V=\{c\in M_n^h(\C):P_V c=c\geq 0,\ \tr(c)=1\}$ and $\sigma_W=\{d\in M_n^h(\C): P_W d=d\geq 0,\ \tr(d)=1\}$, 
	 	satisfy  
	 	$$
	 	\Phi(\sigma_V)\cap\Phi(\sigma_W)\neq\emptyset .
	 	$$
	  \end{enumerate}
\end{theorem}
\begin{proof}
	1. $\Leftrightarrow$ 2. follows after the previous discussion.
	\\
	2. $\Leftrightarrow$ 4. is proved using the comments following the proof of Corollary 3 in  \cite{andruchow_larotonda_recht_varela_Companneras} or the property mentioned in \eqref{eq: adecuacion es cuadrado de dist sigmaV a sigmaW}.
	\\
	1.  $\Leftrightarrow$ 3. is Remark \ref{rem: equiv soporte igualdad diags}.
\end{proof}

\begin{proposition}\label{prop: Srs es cerrado en las banderas}
	The set of supports $\srs$ is closed in the flag manifold $\frs$.
\end{proposition}
\begin{proof}
	Consider a sequence of pairs of supports given by $\{(V_k,W_k)\}_{k\in\N}\subset \srs$ and such that its corresponding orthogonal projections converge. It is apparent that there exist $V$ and $W$ subspaces of $\C^n$ such that $\dim(V)=r$, $\dim(W)=s$, $V\perp W$ and satisfy $\lim_{k\to\infty}P_{V_k}=P_V$ and $\lim_{k\to\infty}P_{W_k}=P_W$, that is, $(V,W)\in \frs$. 
	We only need to prove that the condition \eqref{eq: condicion soporte} holds.
	\\
	Consider for each pair $(V_k,W_k)$ a pair of matrices $(\vu_k,\wu_k)$ that satisfy 
	\begin{equation}\label{eq: igualdad diag vu wu}
	\Phi(\vu_k\ \vu_k^*)=\Phi(\wu_k\ \wu_k^*)
	\end{equation}
	as in Remark \ref{rem: equiv soporte igualdad diags}. 
	\\
	Note that as mentioned in Remark \ref{rem: vectores de condicion soporte pueden ser acotados} we can choose the column vectors of the matrices $\vu_k$ and $\wu_k$ with norm less or equal than one.
	Then using compacity arguments and after taking subsequences we can suppose that the matrices $\vu_k$ are of the same size, and their columns converge to vectors in $V$ that form a matrix $\vu$. Similar arguments can be used for $\wu_k$ to obtain a matrix $\wu$. Since for all $k$ equality \eqref{eq: igualdad diag vu wu} holds, then $\lim_{k\to\infty}\Phi(\vu_k\ \vu_k^*)=\lim_{k\to\infty}\Phi(\wu_k\ \wu_k^*)$ which is $\Phi(\vu \ \vu^*)=\Phi(\wu \ \wu^*)$. Since this is equivalent to the equalities \eqref{eq: condicion soporte} (see Remark \ref{rem: equiv soporte igualdad diags}) then $(V,W)$ is a support. 
\end{proof}

\section{Symplectic interpretation of the map $\Phi$}
Consider the manifold $M=(\C^n)^r$ composed of matrices $\vu$ defined in \eqref{def: v barra}. We denote by $\vu_k$, $k=1,\ldots, n$ the rows of $\vu$ considered as vectors in $\C^r$.

Since $\C^n$ carries a natural symplectic form, so does $M$ (the product form). In this way, $M$ becomes a symplectic manifold. We consider next the left operation action of the unitary group $U(n)$ on $M$. This operation is symplectic. Now we identify the Lie algebra $u(n)$ of $U(n)$ with its dual $u^*(n)$ using the inner product
$\langle A,B\rangle = tr(AB^*).$ 

In this context the moment map $\mu:M\to u^*(n)$ can be computed explicitly:
\begin{equation}
\mu(\vu) =\frac{i}{2}\vu\ \vu^*\quad \text{, for}\quad \vu\in M.
\end{equation}

Observe that the entries of the matrix $\mu(\vu)$ are
\begin{equation}(\mu(\vu))_{k,l}=\frac{i}{2}\langle\vu_k,\ \vu_l\rangle  \ \ \text{, for } k=1\dotsc, n \text{ and }  l=1,\dotsc, n.
\end{equation}
and the entries of the diagonal are
\begin{equation}
(\mu(\vu))_{k,k}=\frac{i}{2}\norm{\vu_k}^2 \ \ \text{, for } k=1,\dotsc, n.
\end{equation}

Finally, observe that for the induced left action of the diagonal unitary matrices $\T^n\subset U(n)$ on $M$, the corresponding moment map $\mu_d$ is obtained as follows
$\mu_d : M \to (t^n)^*$
\[\mu_d(\vu)= \frac{i}{2}
\begin{pmatrix}
\norm{\vu_1}^2\\
\vdots\\
\norm{\vu_n}^2\\
\end{pmatrix},
\]
which is exactly $i/2$ times what was called the \textit{moment of the system}
$\{v^1, \ldots, v^r\}$ in Definition \ref{def: momento del sistema}.

\section{Adequacy of a pair of orthogonal subspaces}\label{sec: adecuacion}

Recall that with ${\mathcal F}_{(r,s)}$ we denote the space pairs of orthogonal subspaces $(V,W)$ of $\C^n$ with $r= \dim V$ and $s=\dim W$. Also 
consider systems $\vu=(v^1,\ldots,v^r):\C^r\to C^n$ as in \eqref{def: v barra} and similarly $\wu=(w^1,\ldots,w^s):\C^s\to C^n$ of vectors in $\C^n$ such that $\im (\vu)\subset V$ and $\im (\wu)\subset W$.


Recall that in Definition \ref{def: momento del sistema} we called $\Phi(\vu\  \vu^*)$ the moment of the system $\vu$ where $\Phi$ is the conditional expectation that associates to any $n\times n$ matrix its diagonal part. This map takes its values in the subalgebra $\mathcal{D}_n$ of diagonal $n\times n$ matrices which will sometimes be identified with $\C^n$. Observe that the map $\vu\mapsto \Phi(\vu\ \vu^*)$ is homogeneous in  the following sense:
\begin{equation}\label{eq: homogeneidad de el peso w}
\Phi\left((\alpha \vu) (\alpha \vu)^*\right) = |\alpha|^2 \Phi(\vu\ \vu^*).
\end{equation}
%


Recall that with this notation $(V,W)$ is called a support (see \ref{eq: condicion soporte} and Theorem \ref{teo: en def. soporte se puede tomar l.i. y ortog.} and Proposition \ref{teo: base ortogonal con mismo momento}) if there is a non trivial pair $(\vu, \wu)$ with $\im(\vu)\subset V$, $\im(\wu)\subset W$ such that
\[
\Phi(\vu\ \vu^*) = \Phi(\wu\ \wu^*)
\]
(here non trivial refers to $\vu\ne (0,\ldots 0)$ and $\wu\ne (0,\ldots 0)$).

Observe that if there is a non trivial pair $(\vu,\wu)$ as before such that $\Phi(\vu\ \vu^*)$ and $\Phi(\wu\ \wu^*)$ are only linearly dependent then choosing $\alpha\in\R$ appropriately we can get $\Phi\left((\alpha\vu)(\alpha\vu) ^*\right) = \Phi(\wu\ \wu^*)$, with $\im(\alpha \vu)\subset V$ so that $(V,W)$ is a support.  

The objective of this section is to define and compute a ``numerical obstruction'' for the pair $(V,W)$ to be a support, i.e. a non negative invariant of $(V,W)$ which vanishes if and only if the pair $(V,W)$ is a support.
We will call this obstruction the \textit{adequacy} of $(V,W)$.

Note that if \eqref{eq: condicion soporte} holds for the vector columns of $\vu$ and $\wu$ then $\tr( \vu\  \vu^*)=\tr( \wu\ \wu^*)$ follows. Then the remark made in \eqref{eq: homogeneidad de el peso w} about the homogeneous nature of $\omega$ allow us to restrict to the space of pairs $(\vu, \wu)$ that are ``normalized'' in the sense that
\[
\tr( \vu\ \vu^*)=1\quad \text{and}\quad \tr( \wu\ \wu^*) = 1.
\]
Observe that in the space $\hom(\C^r,V)$ we have a natural norm given by $\tr(\vu\ \vu^*)^{1/2}$ and the same holds for  $\hom(\C^s,W)$. Therefore if we denote with 
\begin{equation}\label{def: S V y S W}
S_V  \text{ and } S_W \text{ the unit spheres of } \hom(\C^r,V) \text{ and } \hom(\C^s,W)
\end{equation}
respectively, then the selected pairs $(\vu,\wu)$ belong to $S_V\times S_W$.

Finally we define the adequacy of the pair $(V,W)$.
\begin{definition}\label{def: adecuacion}
	Given a pair of non trivial orthogonal subspaces $V, W \subset \C^n$, its \textbf{adequacy} is 
	defined as the number 
	\begin{equation}\label{def: definicion de adequacy}
	\delta(V,W) 
	= \inf\left\{|| \Phi(\vu\  \vu^*)-\Phi(\wu\  \wu^*)||^2:(\vu,\wu)\in S_V\times S_W\right\}
	\end{equation}
	with $S_V$ and $S_W$ defined in \eqref{def: S V y S W}.
\end{definition}

Since $S_V\times S_W$ is compact there always exist $(\vu,\wu)$ in $S_V\times 
S_W$ such  that $\delta(V,W)$ is attained. Note that $\delta(V,W)=0$ 
implies that the subspaces $V$ and $W$ form a support (see Definition \ref{def: 
	soporte}).

Next, in order to compute $\delta(V,W)$ we introduce convenient parameters.
\begin{itemize}
	\item First we fix two isometries
	\begin{equation*}\label{eq: isometrias V y W}
	\mathcal V:\C^r\to V,\qquad \mathcal W:\C^s\to W.
	\end{equation*}
	Observe that in particular, $P_V= \mathcal V\mathcal V^*$ and $P_W= \mathcal W\mathcal W^*$ are the orthogonal projections in $\C^n$ onto $V$ and $W$ respectively.
	\item Then any morphism $f:\C^r\to V$ is of the form $f=\vb g$ for $g:\C^r\to \C^r$ a linear map. If we write the polar form $g=a u$ where $a\ge 0$ and $u$ is unitary, we have
	$f=\vb au$. Therefore we observe, in relation to the problem of parametrization: 
	\begin{enumerate}
		\item $\tr (f f^*) = \tr (\vb a^2 \vb^*) = \tr (\vb^* \vb a^2) = \tr (a^2)$ so 
		that $f\in S_V$ if and only if $\tr( a^2) = 1$.

		\item And we have $\Phi(f f^*) = \Phi(\vb a^2 \vb^*)$.
	\end{enumerate}
\end{itemize}
Similar considerations can be done for $\wb$ and $S_W$.

In view of these remarks we parametrize the problem of finding the minimum of $\delta(V,W)$ as follows.

The parameter space will be 
$
\Sigma=\Sigma_r\times\Sigma_s \ \text{, where }
$ 
\begin{equation}
\label{eq: definicion de Sigma sub r}
\Sigma_r=\{a\in M_r^h(\C):    \tr (a^2)=1\} \ \text{ and }
\ 
\Sigma_s=\{b\in M_s^h(\C):    \tr( b^2)=1\}
\end{equation}
are the unit spheres of the self-adjoint matrices (positive or not) of sizes $r\times r$ and $s\times s$.

The function we have to minimize is $F:\Sigma\to [0,+\infty)$, defined by 
\begin{equation}\label{def: funcion F}
F(a,b) = ||\Phi(\vb a^2\vb^*)-\Phi(\wb b^2\wb^*)||^2
\end{equation}
where the norm is given by $\|x\|=\sqrt{Tr(x^*x)}$.
Its minimum value is the adequacy 
\begin{equation}\label{eq: adequacy y funcion F}
\delta(V,W)=\min_{(a,b)\in\Sigma} F(a,b).
\end{equation}

In the next computations, in order to alleviate the notation, we will write
\begin{equation}
\label{def: Delta}
\Delta = \Delta(a,b) = \Phi(\vb a^2 \vb^*) - \Phi(\wb b^2 \wb^*)
\end{equation}
\subsection{The gradient of $F$}
Now we let $a$ vary as a function of a real parameter $t$ and independently $b$ vary as a function of $u$. Then 
\[\frac{\partial F}{\partial t}=2\langle\frac{\partial \Phi(\vb a^2 \vb^*)}{\partial t}, \Delta\rangle\qquad \text{and} \qquad \frac{\partial F}{\partial u}=-2\langle\frac{\partial \Phi(\wb b^2 \wb^*)}{\partial u}, \Delta\rangle.
\]

If we denote with $\frac{d a}{d t} = X\ ,\ \frac{d b}{d u} = Y$, then
\[\frac{\partial F}{\partial t}= 2\langle\Phi(\vb (aX+Xa) \vb^*),\Delta\rangle
\quad \text{ and }\quad \frac{\partial F}{\partial u}=-2\langle\Phi(\wb (bY+Yb) \wb^*),\Delta\rangle.
\]
Here the inner products are traces of products, so using that $\Delta$ is diagonal we can write
\[
\frac{\partial F}{\partial t}= 2\tr(\vb (aX+Xa) \vb^*\Delta)\quad \text{ and }\quad\frac{\partial F}{\partial u}=-2\tr(\wb (bY+Yb) \wb^*\Delta).
\]
Therefore 
\[
\frac{\partial F}{\partial t}=  2\langle aX+Xa,\vb^*\Delta \vb\rangle_{M_r(\C)}
\quad \text{ and }\quad
\frac{\partial F}{\partial u}= -2\langle bY+Yb,\wb^*\Delta \wb\rangle_{M_s(\C)}.
\]
where the inner products now involved are the natural ones in $r\times r$ and $s\times s$ matrices using the corresponding traces.

If in these algebras $M_r(\C)$ and $M_s(\C)$ we consider the operators  
\begin{equation}\label{def: def Sa y Sb}
S_a(X) = aX+Xa,  \ \text{ and } \ S_b(Y) = bY+Yb,
\end{equation}
then its adjoints (for the natural inner products) are precisely $S_{a^*}=S_a$ and $S_{b^*}=S_b$ since $a, b$ are self-adjoint. 
So we can write
\[
\dfrac{\partial F}{\partial t}=  2\langle X,S_a(\vb^*\Delta \vb)\rangle
\quad \text{ and }\quad
\dfrac{\partial F}{\partial u}=  -2\langle Y,S_b(\wb^*\Delta \wb)\rangle
.
\]
Therefore we obtained the following result.
\begin{theorem}
	The gradient of the function $F:\Sigma\to  [0,+\infty)$, 
	$F(a,b) = ||\Phi(\vb a^2\vb^*)-\Phi(\wb b^2\wb^*)||^2$ on the riemannian manifold 
	$\Sigma=\Sigma_r\times \Sigma_s$ at $(a,b)$ (with $\Sigma_r$ and $\Sigma_s$ as 
	in \eqref{eq: definicion de Sigma sub r}) is
	\begin{equation} \label{eq:grad}
	\grad_{(a,b)}F=2\left( S_a(\vb^*\Delta\vb)_{tan}, -S_b(\wb^*\Delta\wb)_{tan} \right)
	\end{equation}
	where the subscript ``$tan$'' refers to the tangential component (to the sphere $\Sigma_r\times\Sigma_s$) of the corresponding vector: $X_{tan}=X-\langle X,a\rangle a$, for $X\in M_r(\C)$ and  $Y_{tan}=Y-\langle Y,b\rangle b$, for $Y\in M_s(\C)$, $\Delta$ is defined in \eqref{def: Delta}, $S_a$, $S_b$ in \eqref{def: def Sa y Sb} and $\vb$, $\wb$ are fixed isometries as in \eqref{eq: isometrias V y W}.
\end{theorem}

\subsection{Approximation of the adequacy $\delta(V,W)$.} \label{algoritmo para calcular adecuacion}
The previous theorem allow us to construct a gradient descent type algorithm to approximate 
the adequacy of a pair of orthogonal subspaces $(V,W)$ in $\C^n$.
\begin{enumerate}
	\item Starting with $V$ and $W$, construct the corresponding isometries 
	$\mathcal{V}\in \C^{n\times r}$ and 
	$\mathcal{W}\in \C^{n\times s}$ defined in \eqref{eq: isometrias V y W} (take 
	for example an orthonormal basis of $V$ and build the matrix $\mathcal{V}$ 
	whose columns are the vectors of that basis, similarly for $\mathcal{W}$).
	\item Choose randomly two positive definite trace one matrices $a_1\in 
	\C^{r\times r}$ and $b_1\in \C^{s\times s}$. 
	\item Then for $i=1,\dots, k$ calculate recursively: 
	\begin{enumerate}
		\item  
		$(a_i,b_i)-\grad_{(a_i,b_i)} F$ using the identity \ref{eq:grad}:
		\begin{eqnarray*}
			(a_i,b_i)-\grad_{(a_i,b_i)} F&=&
			\Big( a_i -2\Big( 
			S_{a_i}(\vb^*\Delta_i\vb)-\tr\left(S_{a_i}(\vb^*\Delta_i\vb) 
			a_i 
			\right)a_i\Big)\ ,\
			\\
			& &\quad
			b_i +2\Big(S_{b_i}(\wb^*\Delta_i\wb)-\tr\left(S_{b_i}(\wb^*\Delta_i\wb) 
			{b_i}\right){b_i}
			\Big)\Big)
		\end{eqnarray*}
		where $S_c(X)= cX+Xc$, and $\Delta_i=\Phi(\vb a_i^2 \vb^*) - \Phi(\wb b_i^2 
		\wb^*)$.
		\item 	Then consider 
		$(\alpha_{i+1},\beta_{i+1})=(a_i,b_i)-\grad_{(a_i,b_i)} 
		F$, and 
		define $a_{i+1}$ and $b_{i+1}$ as its modules with unit norm:
		$$a_{i+1}=\frac{1}{\tr(|\alpha_{i+1}|^2)^{1/2}}|\alpha_{i+1}|
		\ \text{ and } \ b_{i+1}=\frac{1}{\tr(|\beta_{i+1}|^2)^{1/2}}|\beta_{i+1}|$$
		\item If $i+1<k$ go back to step a) and continue the iteration with 
		$a_{i+1}$ and 
		$b_{i+1}$.
		\end{enumerate}
	\item After finishing the $k$ iterations compute $\tr(\Delta_{k+1} 
	\Delta_{k+1})$ 
	to approximate the adequacy $\delta(V,W)$ (see \ref{def: funcion F}).
\end{enumerate}

In Figure \ref{fig: ejemplo uso adecuacion} it is shown the output of several evaluations of the adequacy using the previous procedure on a pair of orthogonal subspaces moved with the multiplication of a curve of unitary matrices. 

\begin{figure}
	\begin{center}
		\epsfig{file=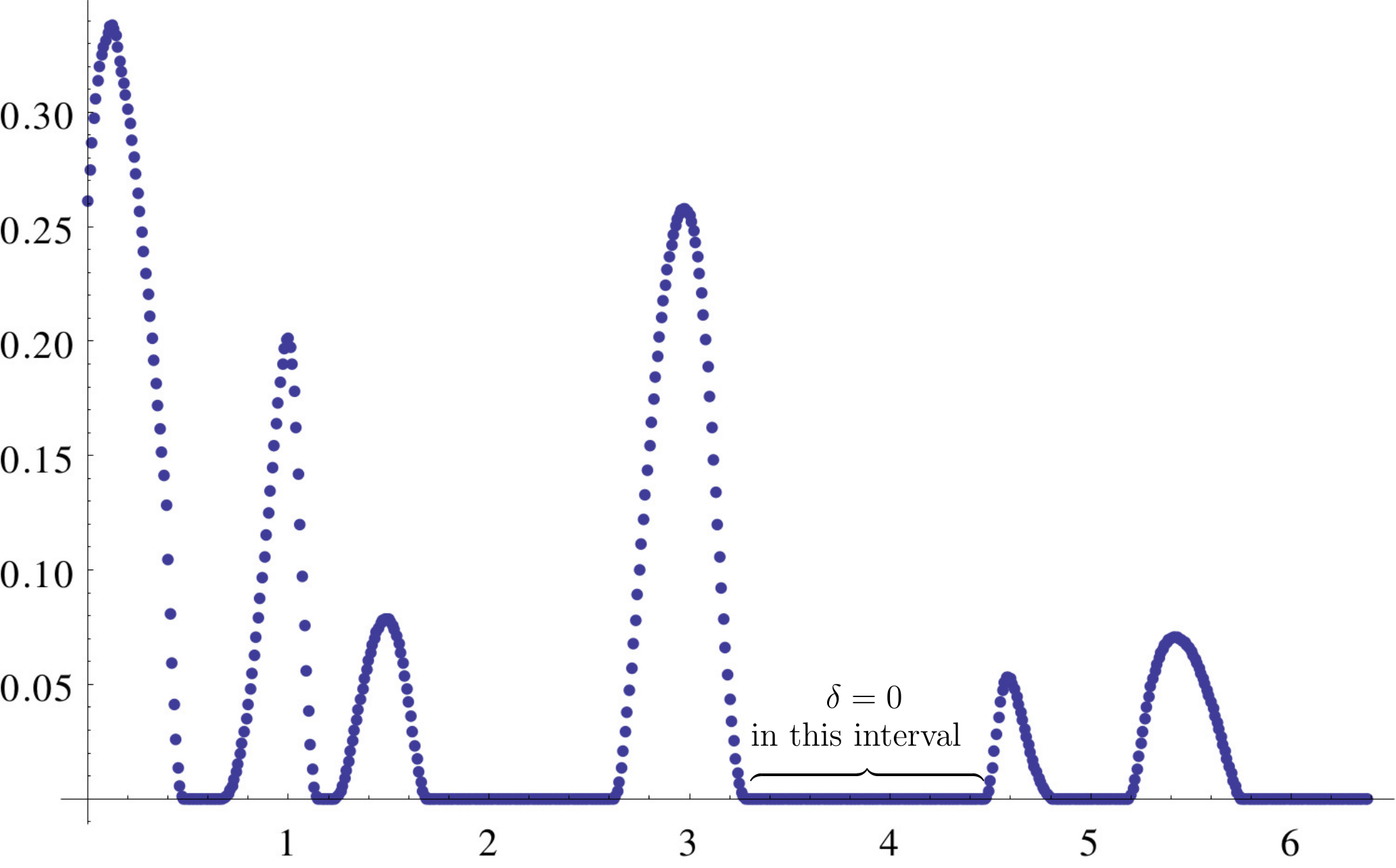,width=10cm}
	\end{center} 
	\caption{
		Plot of the points $(x_j,\delta(V_j,W_j))$, for $x_j=(j/100)$, the subspaces $V_j=e^{i x_j A}(V)$ and $W_j=e^{i x_j A}(W)$, with $j=1,\dots,650$, starting with $V\perp W$, and $A$ a self-adjoint matrix, using the algorithm mentioned in \ref{algoritmo para calcular adecuacion} to calculate the adequacy $\delta$. Observe the intervals where the approximation of the adequacy is null that suggest that for those values of $x_j$ the pairs $(V_j,W_j)$ form a support.}
		\label{fig: ejemplo uso adecuacion}
\end{figure} 
\begin{remark}
	Some of the examples presented in \ref{apendice A: ej 3x3}, 
	\ref{apendice B: ej 4x4} and \ref{apendice C: ej 5x5} were obtained using  the previous algorithm to approximate the adequacy.
\end{remark}
\medskip

\subsection{The critical points of $F$}
The point $(a,b) \in \Sigma$ is critical for $F$ if and only if
$S_a(\vb^*\Delta\vb)$ is normal to $\Sigma_r$ and $S_b(\wb^*\Delta\wb)$ is normal to $\Sigma_s$. Then we can state the following result.
\begin{theorem}
	The point $(a,b)\in \Sigma=\Sigma_r\times\Sigma_s$ is critical for $F$ if and only if
	\begin{equation}\label{eq: condiciones punto critico F}
	\left\{\begin{matrix}
	S_a(\vb^*\Delta\vb) = \lambda a\\
	S_b(\wb^*\Delta\wb) = \mu b
	\end{matrix}\right. 
	\ \ \text{, for  } \lambda, \mu\in \R.
	\end{equation}
\end{theorem}
\subsection{Analysis of the conditions \eqref{eq: condiciones punto critico F}}
Suppose that we have operators $c\ge 0$, $u$ self-adjoint and $cu+uc=\eta c$ where $\eta\in\R$. Then the following commutation rule holds
\[
\left\{
\begin{array}{rcl}c u &= (\eta - u)c\\
u c & = c(\eta - u)
\end{array}\right..
\]
Then, $u$ commutes with $c$ and we have $uc=cu=\frac{\eta}{2}c$.
The previous comments allow us to state the next result.
\begin{theorem}\label{teo: condiciones necesarias de a y b positivos para ser punto critico de F}
	In a critical point $(a,b)$ of $F$ as in \eqref{eq: condiciones punto critico F} where $a\ge 0$ and $b\ge 0$ then
	$\vb^*\Delta\vb$ commutes with $a$ and $(\vb^*\Delta\vb)a = \frac{\lambda}{2} a$ and also $\wb^*\Delta\wb$ commutes with $b$ and   
$(\wb^*\Delta\wb)b = \frac{\mu}{2} b.
$
\end{theorem}
\textbf{Remark:} In these notes we are interested in the minimum value of $F$ on $\Sigma$. Since $(a,b)\in \Sigma$ implies $(|a|,|b|)\in \Sigma$, because $a^2=|a|^2$, $b^2=|b|^2$ if $a,b$ are hermitian, and $F(a,b)=F(|a|,|b|)$ it is clear that the minimum of $F$ is attained on some $(a,b)$ with $a\ge0$ and $b\ge 0$. 

\subsection{The Hessian of the map $F$}
Recall the expression of $\grad_{(a,b)}F$ obtained in \eqref{eq:grad} and the 
definition of $S_a$ and $S_b$ in \eqref{def: def Sa y Sb} for $a\in\Sigma_r$ and 
$b\in\Sigma_s$ (see \eqref{eq: definicion de Sigma sub r}).

We write $V_{tan}$ to denote the tangential part of $V\in M^h_r(\C)$ of the sphere $\Sigma_r$ when $V$ is considered as a tangent vector at a point of $\Sigma_r$ (correspondingly for $W\in M^h_s(\C)$ and $\Sigma_s$).
\begin{figure}
	\begin{center}
		\epsfig{file=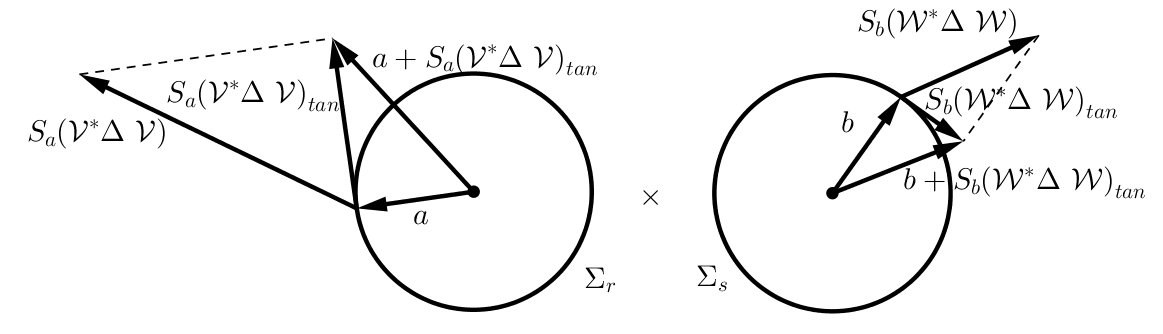,width=12cm}
	\end{center} 
	\caption{Increments on the tangents of both spheres $\Sigma_r$ and $\Sigma_s$ in the direction of the 
		gradient used in the approximation algorithm for the adequacy $\delta(V,W)$.\label{fig: tangentes esferas}}
\end{figure} 
Let us denote with  $\pi_r: M^h_r(\C)\to T(\Sigma_r)_a$ and   $\pi_s: 
M^h_s(\C)\to T(\Sigma_s)_b$
\begin{gather}
\pi_r(V) =V_{tan} =  V - \langle V,a\rangle a , \text{ for } a\in\Sigma_r , 
V\in M_r^h(\C)\\
\pi_s(W) =W_{tan} =  W - \langle W,b\rangle b , \text{ for } b\in\Sigma_s , W 
\in M_r^h(\C).
\end{gather}
Recall also that in a riemannian manifold, the Hessian of a function $U$ at a critical point is given by
\begin{equation}H(U)(Z,W) = \langle D_Z \grad U, W\rangle\label{eq:hessiano}\end{equation}
where $D$ denotes the covariant derivative of the Levi-Civita connection of the metric.
Finally recall that the covariant derivative in our case is the tangent projection of the ``ambient'' derivative.

In the computations below we will need expressions for the derivatives 
$\partial_X$ and $\partial_Y$ in the directions   $X\in T(\Sigma_r)_a$ and 
$Y\in T(\Sigma_s)_b$ respectively of the projections $\pi_r$ and $\pi_s$.


Recall that in \eqref{eq:grad} we calculated 
$$
\grad_{(a,b)}F = 2\left(\pi_r\left( 
S_a(\vb^*\Delta \vb)\right), 
-\pi_s\left( S_b(\wb^*\Delta \wb)\right) \right)\ 
$$
 where $\Delta = \Phi(\vb a^2 \vb^*) - \Phi(\wb b^2 \wb^*).
$
In order to calculate \eqref{eq:hessiano} we can use that
$D_{(X,Y)}=D_{(X,0)+(0,Y)} = D_{(X,0)} + D_{(0,Y)}.$


Then the covariant derivative of $\pi_r(S_a(\vb^*\Delta \vb))$ is given by
\begin{equation}
\label{eq: primera formula der covariante primer componente de F}
\begin{split}
D_{(X,Y)}  & (S_a(\vb^*\Delta \vb)- \langle S_a(\vb^*\Delta \vb) , 
a\rangle a )= \\
 = &  S_X(\vb^*\Delta \vb) + S_a(\vb^*(\partial_X\Delta+\partial_Y\Delta) \vb) 
\\ 
&-   \langle S_a(\vb^*(\partial_X\Delta+\partial_Y\Delta) \vb),a\rangle a
-  2\langle S_X(\vb^*\Delta \vb),a\rangle a \\
&-  \langle S_a\left( \vb^*\Delta \vb\right),a \rangle X 
\end{split}
\end{equation}
where we have used that $S_a$ and $S_X$ are self-adjoint and $S_a(X)=S_X(a)$. 

The covariant derivative of $\pi_s(S_b(\wb^*\Delta \wb))$ can be calculated 
similarly. 

Observe that 
\begin{equation}
\begin{split}
\label{eq: formula de derivada de Delta respecto a X}
\partial_X(\Delta)&=\partial_X\left(\Phi\left( \vb a^2\vb^*-\wb 
b^2\wb^* 
\right)\right)\\
&=\Phi(\vb (aX+Xa)\vb^*)=\Phi(\vb S_a(X)\vb^*)
\end{split}
\end{equation} 
and $
\partial_Y(\Delta)=  -\Phi(\wb S_b(Y)\wb^*)
$. Then using that $\partial_X\Delta+\partial_Y\Delta$ is diagonal
\begin{align*}
\langle S_a(\vb^*(\partial_X\Delta+\partial_Y\Delta) \vb) , X\rangle& =
\langle (\partial_X\Delta+\partial_Y\Delta) ,  \vb S_a(X)\vb^*\rangle \\
&= \langle (\partial_X\Delta+\partial_Y\Delta) ,  \Phi\left(\vb 
S_a(X)\vb^*\right)\rangle\\
& =
\langle (\partial_X\Delta+\partial_Y\Delta) ,  \partial_X\Delta\rangle 
\end{align*}
where we have used in the last equality the formula obtained in \eqref{eq: formula de derivada de Delta respecto a X} for $\partial_X\Delta$. Similarly we 
can prove that 
$\langle S_b(\wb^*(\partial_X\Delta+\partial_Y\Delta) \wb) , Y\rangle =
-\langle (\partial_X\Delta+\partial_Y\Delta) ,  \partial_Y\Delta\rangle 
$. Then
\begin{align}
\label{eq: formula con derivadas resp a X y a Y}
\langle S_a(\vb^*(\partial_X\Delta&+\partial_Y\Delta) \vb) , X\rangle -
\langle S_b(\wb^*(\partial_X\Delta+\partial_Y\Delta) \wb) , Y\rangle= 
\nonumber\\
& =
\langle (\partial_X\Delta+\partial_Y\Delta) ,  \partial_X\Delta\rangle  +
\langle (\partial_X\Delta+\partial_Y\Delta) ,  \partial_Y\Delta\rangle 
\\&=
\langle (\partial_X\Delta+\partial_Y\Delta) ,  
(\partial_X\Delta+\partial_Y\Delta)\rangle\nonumber\\
&=\|\partial_X\Delta+\partial_Y\Delta\|^2
\nonumber
\end{align}

Finally, using the expression (\ref{eq:hessiano}) 
for the quadratic form $H(F)\left((X,Y),(X,Y)\right) $ and 
\eqref{eq: primera formula der covariante primer componente de F} we obtain that
\begin{equation*}\begin{split}
H(F)\left((X,Y),(X,Y)\right)&= \langle D_{(X,Y)}\left(\pi_r(S_a(\vb \Delta 
\vb^*))\right),X\rangle -
\langle D_{(X,Y)}\left(\pi_s(S_b(\wb^*\Delta \wb))\right), Y\rangle
\\
&=||\partial_X\Delta+\partial_Y\Delta||^2 + 2\left(\langle \vb^*\Delta \vb 
X,X\rangle-\langle \wb^*\Delta \wb Y,Y\rangle\right)\\
&-\langle S_a(\vb^*\Delta \vb),a\rangle ||X||^2
+\langle S_b(\wb^*\Delta \wb),b\rangle ||Y||^2
\end{split}
\end{equation*}
where we have used that $a\perp X$, $b\perp Y$, $  \langle 
S_X(\vb^*\Delta \vb),X\rangle= 2 \langle \vb^*\Delta \vb 
X,X\rangle$ and $\langle S_Y(\wb^*\Delta \wb),Y\rangle= 2 \langle \wb^*\Delta 
\wb Y,Y\rangle$.
We could simplify the expression of the Hessian even more using that $\langle 
S_a(\vb^*\Delta \vb),a\rangle=2\langle \vb^*\Delta \vb, a^2\rangle$ and  
$\langle S_b(\wb^*\Delta \wb),b\rangle=2\langle \wb^*\Delta \wb,b^2\rangle$ to 
obtain the following result
\begin{theorem}
	\label{teo: hessiano}
	The Hessian of the map $F:\Sigma_r\times\Sigma_s\to\R_{\geq 0}$ (see \eqref{def: funcion F} and \eqref{eq: definicion de Sigma sub r}) for  $X\in T(\Sigma_r)_a$ and 
	$Y\in T(\Sigma_s)_b$ at a critical point $(a,b)$ can be calculated as
	\begin{equation*}
	\begin{split} 
	H(F)\left((X,Y),(X,Y)\right) 
	=||\partial_X\Delta+\partial_Y\Delta||^2 + 2&\left(\langle \vb^*\Delta \vb ,X^2-a^2 \|X\|^2\rangle\right. \\
	&\ \ -\left.\langle \wb^*\Delta \wb  
	,Y^2-b^2\|Y\|^2\rangle\right). 
	\end{split}
	\end{equation*}
		
\end{theorem}

\section{A geometric interpretation of the adequacy}\label{sec: interpret geom de adecuacion}
Let $V$ and $W$ be two orthogonal subspaces of $\C^n$ with dim$(V)=r$ and dim$(W)=s$ as before.

In this section we distinguish three subalgebras of $M_n(\C)$.
\begin{enumerate}
\item  $\mathcal{D}_n\subset M_n(\C)$ the subalgebra of diagonal matrices and
$
\Phi:M_n(\C)\to \mathcal{D}_n
$
the conditional expectation that associates to the matrix $m$ its diagonal part $\Phi(m)$ as before. Observe that $\Phi$ is an orthogonal projection for the natural Hilbert structure of $M_n(\C)$. We have for $m\in M_n(\C)$
\begin{equation}
\label{def: def de Phi esp cond a diag}
\Phi(m) = \sum_{k=1}^n p_kmp_k
\end{equation}
where $p_k$ is the orthogonal projection onto the $k$-axis of $\C^n$.
\item
%
We denote with $M_n(V)\subset M_n(\C)$ the subalgebra of the endomorphisms $x$ 
of $\C^n$ which commute with $P_V=\vb\vb^*$ (for $\vb:\C^r\to V\subset\C^n$ an 
isometry with range $V$) and verify $P_Vx=x$. Observe that $M_n(V)$ is a 
C$^*$-subalgebra of $M_n(\C)$ with identity $P_V$.

Also the map $\pb_V:M_n(\C)\to M_n(\C)$ defined by
\[
\pb_V(m) = P_VmP_V
\]
satisfies the requirements of a conditional expectation in $M_n(\C)$ with image 
$M_n(V)$, except for the fact that $\pb_V(I)=P_V\neq I$.
Finally
\[
\mathcal I_V:M_r(\C)\to M_n(V)\subset M_n(\C)\ \ ,\ \  \mathcal I_V(a)=\vb a\vb^*
\]
defines an isomorphism of C$^*$-algebras between $M_r(\C)$ and $M_n(V)$.
\item Similarly we denote $M_n(W)$, $P_W$, $\pb_W$ and $\mathcal I_W$ 
related to the subspace $W$. Notice that $M_n(V)$ and $M_n(W)$ are orthogonal 
in $M_n(\C)$ for the Hilbert space structure and also in the sense that
\[
ab = ba = 0 \text{ for } a\in M_n(V) \text{ and } b\in M_n(W).
\]
\end{enumerate}

Now we analyze the optimization problem of computing the adequacy of $(V,W)$ in this context.

We denote with $M^h_n(V)$ the self-adjoint part of $M_n(V)$. The function 
$a\mapsto \vb a^2 \vb^*$ maps bijectively the positive part $\Sigma_r^+$ of the 
unit sphere $\Sigma_r= \{a\in M_r^h(\C):    \tr (a^2)=1\}$ (see 
\eqref{eq: definicion de Sigma sub r}) onto the set
\begin{equation}\label{def: def de sigmaV}
\sigma_V = \{c\in M^h_n(V):c\geq 0\text{ and }\tr c=1\}.
\end{equation}
Note that if $c\in\sigma_V$, then $\vb^*c^{1/2}\vb$ lies in $\Sigma_r^+$, where $c^{1/2}$ is the positive square root of the operator $c$. Similar considerations apply to $W$ and we can define the corresponding $\sigma_W$.

Recall that the minimum of the function $F$ (the adequacy of the pair $(V,W)$, see \eqref{eq: adequacy y funcion F}) is attained, among other points, at some $(a,b)\in\Sigma_r\times\Sigma_s$ where $a\ge0$ and $b\ge0$. Therefore the adequacy can be obtained as the square of the distance of the set $\Phi(\sigma_V)$ to the set $\Phi(\sigma_W)$
\begin{equation}
\label{eq: adecuacion es cuadrado de dist sigmaV a sigmaW}
\delta(V,W)=\left(\text{dist}\left(\Phi(\sigma_V), \Phi(\sigma_W)\right)\right)^2
\end{equation}

Now we describe the set $\Phi(\sigma_V)$ (and similarly $\Phi(\sigma_W)$). Clearly $\sigma_V$ is a convex compact set in $M^h_n(V)$ and therefore $\sigma_V$ is the convex hull of the set $\sigma_V^e$ of its extremal points. Since $\mathcal I_V:M_r(\C)\to M_n(V)\subset M_n(V)$ is an isomorphism of C$^*$-algebras, the set $\sigma_V^e$ consists of the projections $p$ of rank one in $M_n(V)$. Now these projections $p$ are obtained as follows
\[p=uu^*, \quad \text{ with } u  \ \text{ a unit vector in $V$.}
\]
In this case the diagonal of $p$ coincides with
$\Phi(p) = \text{diag}
\left( |u_1|^2, \dots , |u_n|^2\right).
$

Let us denote by $\Sigma_V\subset \C^n$ the unit sphere of $V$ and correspondingly by $\Sigma_W\subset \C^n$ the unit sphere of $W$. 

Also define $m:\C^n\to\R^n$ by  
\begin{equation}\label{def: funcion m momento de vector complejo }
m(v) \simeq \Phi(v v^*)
\end{equation}
where we identified the diagonal $\Phi(v v^*)$ with the vector $((vv^*)_{1,1},\dots,(vv^*)_{n,n})\in \R^n_{\geq 0}$.
Then we can state the next result.
\begin{theorem}\label{teo: igualdad diagonal de tr 1 posit en V es igual a la caps conv esfera de V}
 If $m$ is as in \eqref{def: funcion m momento de vector complejo }, $\Sigma_V$ is the unit sphere of the subspace $V$ and $\mathop{co}(m(\Sigma_V))$ is the convex hull of the set $m(\Sigma_V)$, then
$$
\Phi(\sigma_V)=\mathop{co}(m(\Sigma_V))
$$
for $\Phi$ defined in \eqref{def: def de Phi esp cond a diag} and $\sigma_V$ in \eqref{def: def de sigmaV}.
\end{theorem}

\begin{proof} Since $\Phi$ is linear, $\Phi(\sigma_V)$ is a convex compact set in $\R^n$. Therefore, $\Phi(\sigma_V)$ is the convex hull of its extremal set. But it is well known that the extremal set of $\Phi(\sigma_V)$ is contained in the image $\Phi(\sigma_V^e)$ which is $m(\Sigma_V)$. Therefore $\Phi(\sigma_V)$ is included in the convex hull of $m(\Sigma_V)$.

The inclusion $m(\Sigma_V)\subset \Phi(\sigma_V)$ implies that $co(m(\Sigma_V)) \subset \Phi(\sigma_V)$ which proves the equality.
\end{proof}
\begin{remark}
	Note that in general the set of extremal points of $\Phi(\sigma_V)$ is strictly included in $\Phi(\sigma_V^e)=m(\Sigma_V)$.
\end{remark}

\begin{remark}  If $S^{2r-1}$ denotes the unit sphere in $\C^r$ then, since $\Sigma_V =\vb S^{2r-1}\vb^*$, we can replace $m(\Sigma_V)$ with $m(\vb S^{2r-1}\vb^*)$ in the previous theorem.
\end{remark}


\section{On the critical points of the function $F$}\label{sec: On the critical points of the function F}
The results of the previous section motivates the study of minimum values of $F:\Sigma=\Sigma_r\times\Sigma_s\to\R_{\geq 0}$ (see \eqref{def: funcion F}) attained at extremal points of the sets $\Phi(\sigma_V)$ and $\Phi(\sigma_W)$. In this section we describe critical points of $F$ under the assumption that they are attained on pairs of one dimensional projections. This would always be the case if the sets $\Phi(\sigma_V)$ and $\Phi(\sigma_W)$ were strictly convex as seen in all the examples we examined where none of the vectors of the standard basis belong to either subspace.

We assume the following:
\begin{enumerate}
\item $(a,b)\in \Sigma$ is a critical point for $F$
\item $a$ and $b$ are one dimensional projections in $\C^r$ and $\C^s$ respectively. 
\item We choose $\tilde a\in\C^r$ and  $\tilde b\in\C^s$ unit vectors such that
\begin{equation*}
a(x) = \langle x,\tilde a\rangle\tilde a\ ,\quad x\in\C^r \text{ , and }
b(y) = \langle y,\tilde b\rangle\tilde b\ ,\quad y\in\C^s.
\end{equation*}
\item We denote with
\begin{equation*}
x_k = \vb^* e_k \text{ , and }
y_k = \wb^* e_k,\text{  for  } k=1,\ldots, n,
\end{equation*}
 where $e_k$ are the standard base vectors and $\vb$, $\wb$ some fixed isometries as in \eqref{eq: isometrias V y W}.
\item We denote $\alpha_k = \langle\tilde a,x_k\rangle$ and $\beta_k = \langle\tilde b,y_k\rangle$. Then, using that $\vb\tilde a=\sum \langle \vb \tilde a , e_k\rangle e_k$, and that therefore $\tilde a=\vb^*\vb \tilde a =  \sum \langle \vb \tilde a , e_k\rangle\vb^* e_k $ we can conclude that $\tilde a=\sum\alpha_k x_k$. Similarly $\tilde b=\sum\beta_k y_k$ can be obtained.

\item Since $a, b\geq 0$, after some computations follows that the pair $(a,b)$ is a critical point for the function $F$ (see Theorem \ref{teo: condiciones necesarias de a y b positivos para ser punto critico de F}) if and only if
\begin{equation} \label{eq: condicion alfa beta lambda equis mu ygriega}
\sum_{k=1}^n(|\alpha_k|^2 - |\beta_k|^2 - \lambda/2)\alpha_k x_k = 0 
= \sum_{k=1}^n(|\alpha_k|^2 - |\beta_k|^2 - \mu/2)\beta_k y_k.
\end{equation}
\end{enumerate}
Observe that $\sum |\alpha_k|^2=1$. In fact $\alpha_k = \langle\vb \tilde a,e_k\rangle$ and since $\vb$ is an isometry, $||\vb \tilde a||=1$. Similarly $\sum|\beta_k|^2=1$.

Now we turn the analysis of equations \eqref{eq: condicion alfa beta lambda equis mu ygriega}. First notice that there exist non trivial complex combinations of the form
\[\sum_{k=1}^n\xi_k x_k = 0 \text{ , and } \sum_{k=1}^n \eta_ky_k=0\]
because $r = \dim V<n$ and $s = \dim W<n$.

For each of such pairs, $\xi_1,\eta_1;\dots ;\xi_n,\eta_n$ consider the system
\begin{equation}\label{eq: condicion alfa beta lambda xi mu eta}
\alpha_k|\alpha_k|^2 - (\lambda/2+|\beta_k|^2)\alpha_k - \xi_k=0, \text{ and } 
\beta_k| \beta_k|^2 + (\mu/2-|\alpha_k|^2)\beta_k + \eta_k=0  
\end{equation}
obtained from \eqref{eq: condicion alfa beta lambda equis mu ygriega} identifying each coefficient with the corresponding $\xi_k$ and $\eta_k$.

Next we multiply the first equation of \eqref{eq: condicion alfa beta lambda xi mu eta} by $\varphi_k$ and the second by $\psi_k$ (with $|\varphi_k|=1$ and $|\psi_k|=1$) so that each $\sigma_k = \varphi_k\xi_k$ is real and $\tau_k = \psi_k\eta_k$ is real. Defining $s_k = \varphi_k\alpha_k$ and $t_k = \psi_k\beta_k$ we get from equations \eqref{eq: condicion alfa beta lambda xi mu eta}
\begin{equation}\label{eq: condiciones lambda te ese sigma tau}
s_k^3 - (\lambda/2+t_k^2)s_k - \sigma_k=0, \text{ and }
t_k^3 + (\mu/2-s_k^2)t_k + \tau_k=0 
\end{equation}
and all the coefficients of these equations are real numbers.

In fact multiplying the first equation in \eqref{eq: condicion alfa beta lambda equis mu ygriega} by $\tilde a$ and the second by $\tilde b$ we get
\begin{equation}\label{eq: ecuaciones alfa beta lambda mu}
\sum_{k=1}^n(|\alpha_k|^2 - |\beta_k|^2)|\alpha_k|^2 = \lambda/2 \text{ , and } 
\sum_{k=1}^n(|\alpha_k|^2 - |\beta_k|^2)|\beta_k|^2 = \mu/2
\end{equation}
which shows that $\lambda$ and $\mu$ are real and moreover $\lambda\ge\mu$ because
\[\frac{\lambda-\mu}{2} = \sum (|\alpha_k|^2-|\beta_k|^2 )^2.\]

In terms of $s_k$ and $t_k$ equations \eqref {eq: ecuaciones alfa beta lambda mu} can be rewritten in the form
\begin{equation}\label{eq: ecuaciones ese te lambda mu}
\sum_{k=1}^n(s_k^2 - t_k^2)s_k^2 = \lambda/2 \text{ , and } 
\sum_{k=1}^n(s_k^2 - t_k^2)t_k^2 = \mu/2.
\end{equation}
The set of equations \eqref{eq: condiciones lambda te ese sigma tau} and \eqref{eq: ecuaciones ese te lambda mu} form a complete system of 2n+2 equations with 2n+2 unknowns.
%
%

\section{Characterization of critical points of $F$}\label{sec: Characterization of critical points of F}

Based on the discussion of the preceding paragraphs we state the following theorem.

\begin{theorem}
Let $\mathcal V: \C^r\to\C^n$ and $\mathcal W: \C^s\to\C^n$ be fixed isometries such that R$(\mathcal V)=V\perp$ R$(\mathcal W)=W$,  $\{e_k\}_{k=1}^n$ be the standard basis of $\C^n$ and $a, b$ be unidimensional projections in $\C^r$ and $\C^s$ respectively. Then the following statements are equivalent,
\begin{enumerate}
 \item[i) ] the pair $(a,b)$ is a critical point of the map $F$ (defined in \eqref{def: funcion F}),
\item [ii) ] there exists a pair of unitary vectors $(\tilde a, \tilde b)\in \C^r\times \C^s$ such that $a= \langle \cdot , \tilde a\rangle \tilde a$ and $b= \langle \cdot , \tilde b\rangle \tilde b$, and $(\tilde a,\tilde b)$ satisfy equations \eqref{eq: condicion alfa beta lambda equis mu ygriega} for 
$\alpha_k =\langle \tilde a, \mathcal{V}^* e_k\rangle$ and $\beta_k =\langle \tilde b, \mathcal{W}^* e_k\rangle$ for $\kn$,
and  

\item[iii) ] there exists a pair of unitary vectors $(\tilde a, \tilde b)\in \C^r\times \C^s$ such that
$\tilde a= \sum_{k=1}^n  \overline\varphi_k s_k \mathcal{V}^*e_k$ and $\tilde b= \sum_{k=1}^n  \overline\psi_k t_k \mathcal{W}^*e_k$, where 
\begin{enumerate}
\item $s_k, t_k\in\R$, for $\kn$ and $\sum_{k=1}^n  s_k^2=1=\sum_{k=1}^n  t_k^2$,
\item $\varphi_k, \psi_k\in\C$, $|\varphi_k|=|\psi_k|=1$, for $\kn$
\item $\|\sum_{k=1}^n  \overline\varphi_k s_k \mathcal{V}^*e_k\|=1=\|\sum_{k=1}^n  \overline\psi_k \ t_k\mathcal{W}^*e_k\|$

\item there exists  $\sigma_k, \tau_k\in\R_{\geq 0}$ such that 
$\displaystyle{ 
	 \left\{\begin{array}{l}
	 \sum_{k=1}^n \overline{\varphi}_k \sigma_k \mathcal{V}^* e_k = 0\\
	 \sum_{k=1}^n \overline{\psi}_k \tau_k \mathcal{W}^* e_k=0,
	 \end{array}\right.}$
\item 
and $s_k, t_k \in \R$, for $\kn$ are solutions of the systems
\[
 \left\{\begin{array}{rcl}
s_k^3- \Big(\big(\sum_{j=1}^n (s_j^2-t_j^2)s_j^2\big)+t_k^2 \Big)s_k+\sigma_k&=&0\\
t_k^3+\Big(\big(\sum_{j=1}^n (t_j^2-s_j^2)t_j^2\big)-s_k^2\Big)t_k+\tau_k&=&0\\
        \end{array}
        \right.
 \]
\end{enumerate}
\end{enumerate}
\end{theorem}

\begin{proof}
The equivalence i) $\Leftrightarrow$ ii) has been discussed in the previous section.

ii) $\Rightarrow$ iii) has also been proved at the end of the previous section.

Let us consider the implication iii) $\Rightarrow$ ii).

If we define $\lambda$ and $\mu$ with $
 \left\{\begin{array}{rcl}
\lambda/2&=&\sum_{k=1}^n (s_k^2-t_k^2)s_k^2\\
\mu/2&=&\sum_{k=1}^n (t_k^2-s_k^2)t_k^2\\
        \end{array}
        \right.
 $
 then $\lambda$, $\mu$, $s_k, t_k$ (for $\kn$) satisfy \eqref{eq: ecuaciones ese te lambda mu}. 
Moreover, iii) (e) implies that they also satisfy \eqref{eq: condiciones lambda te ese sigma tau}.

 Let us now define $\alpha_k=\overline{\varphi}_k s_k$, $\beta_k=\overline{\psi}_k t_k$, for $\kn$, and observe that the equations iii) (c) $\|\sum_{k=1}^n  \overline\varphi_k s_k \mathcal{V}^*e_k\|=1=\|\sum_{k=1}^n  \overline\psi_k \ t_k\mathcal{W}^*e_k\|$ are equivalent to $\sum_{k=1}^n  \overline\varphi_k s_k  e_k= \sum_{k=1}^n  \alpha_k   e_k\in V$ and $\sum_{k=1}^n  \overline\psi_k t_k  e_k=\sum_{k=1}^n  \beta_k  e_k\in W$. Then if we define 
 $$
 \begin{array}{l}
 \tilde a= \sum_{k=1}^n  \overline\varphi_k s_k \mathcal{V}^*e_k =\sum_{k=1}^n  \alpha_k \mathcal{V}^*e_k\\
 \tilde b= \sum_{k=1}^n  \overline\psi_k t_k \mathcal{W}^*e_k= \sum_{k=1}^n  \beta_k \mathcal{W}^*e_k
 \end{array}
 $$
 follows that
 $$
 \begin{array}{l}
 \vb\tilde a  =\sum_{k=1}^n  \alpha_k  e_k\ \  \Rightarrow \ \ \alpha_k =\langle \tilde a, \mathcal{V}^* e_k\rangle\\
 \wb\tilde b = \sum_{k=1}^n  \beta_k  e_k\ \  \Rightarrow\ \  \beta_k =\langle \tilde b, \mathcal{W}^* e_k\rangle
 \end{array}
 $$
 (since $  \sum_{k=1}^n  \alpha_k   e_k\in V$ and $ \sum_{k=1}^n  \beta_k  e_k\in W$).
 
 Now if we define $\xi_k=\overline{\varphi}_k \sigma_k$ and  $\eta_k=\overline{\psi}_k \tau_k$, for $\kn$,  then iii)(e) implies that equations \eqref{eq: condicion alfa beta lambda xi mu eta} are satisfied  and therefore equations \eqref{eq: condicion alfa beta lambda equis mu ygriega} are also satisfied with $\alpha_k =\langle \tilde a, \mathcal{V}^* e_k\rangle$ and $\beta_k =\langle \tilde b, \mathcal{W}^* e_k\rangle$. Then statement ii) holds.
 \end{proof}

\section{Supports that have neighborhoods of $\srs$ in $\mathcal{F}_{(r,s)}$}\label{sec: soportes con entornos de soportes}

Recall that with $\mathcal{F}_{(r,s)}$ we denote the set of pairs $(V,W)$ of orthogonal subspaces $V$ and $W$ of $\C^n$ such that $\dim(V)=r$ and $\dim(W)=s$. See Section \ref{sec: preliminares} for its relation with flag manifolds.

In this section we study the existence of supports $(V,W)\in\srs$ that belong to an open neighborhood of $\mathcal{F}_{(r,s)}$ formed entirely of supports in $\srs$.
\begin{remark}
	Note that in general, a support $(V,W)$ of $\C^n$ in the flag 
	$\mathcal{F}_{(r,s)}$, is not necessarily an interior point of 
	$\mathcal{F}_{(r,s)}$. Consider for example two orthogonal one dimensional 
	subspaces $V=\text{gen}\{v\}$ and $W=\text{gen}\{w\}$ that form a support 
	in $\C^n$ ($n\geq 3$). 
	Then their generators must satisfy $|v_i|=|w_i|$ for $i=1,\dots, n$. 
	Suppose that $v_1\neq 0$, and $ v_1=\rho e^{i\theta}$, $w_1=\rho 
	e^{i\beta}$ with $\rho=|v_1|=|w_1|$ and $\theta,\beta \in[0,2\pi)$,
	Then for $\varepsilon>0$ consider small perturbations $v_\varepsilon$ and 
	$w_\varepsilon$ with their first coordinates 
	$(v_\varepsilon)_1 =\rho/(1+\varepsilon)e^{i\theta}$ and 
	$(w_\varepsilon)_1 =\rho(1+\varepsilon)e^{i\beta}$ and the rest equal to 
	those of $v$ and $w$. Then $\langle v_\varepsilon,w_\varepsilon\rangle = 
	\langle v,w\rangle  =0$ but $|(v_\varepsilon)_1 |=\rho/(1+\varepsilon)\neq  
	\rho (1+\varepsilon)=|(w_\varepsilon)_1 |$ for $\varepsilon>0$. If we 
	denote 
	with $V_\varepsilon$ and $W_\varepsilon$ the subspaces generated by  
	$v_\varepsilon$ and $ w_\varepsilon$ respectively, the previous 
	calculations prove that there exist pair of subspaces $(V_\varepsilon, 
	W_\varepsilon)$ in the flag $\mathcal{F}_{(1,1)}$ that do not form a 
	support and that they can be chosen as close to $(V,W)$ as desired (taking 
	$\varepsilon\to 0$). Therefore $(V,W)$ is not an 
	interior point of 
	$\mathcal{F}_{(1,1)}$.
\end{remark}
\begin{theorem}\label{teo: los soportes forman un abierto}
	Let $n\in\N$, $n\geq 3$. Then, there exists a support $(V_n,W_n)$  in $\C^n$ 
	that is an interior point of the flag 
	$\mathcal{F}_{(r,s)}$ for certain $r$, $s$ $<n$.
\end{theorem}
\begin{proof}
	We will use the examples described in the appendices in the cases $n=3$, $n=4$ and $n=5$ where some cases of supports 
	that are interior points of the flags $\mathcal{F}_{(2,1)}$, $\mathcal{F}_{(2,1)}$ and 
	$\mathcal{F}_{(3,1)}$ are shown.

	Consider now 
	the supports $(V,W)$ of $\C^3$, $\C^4$ and $\C^5$ described in appendices 
	\ref{apendice A: ej 3x3},  \ref{apendice B: ej 4x4} and \ref{apendice C: ej 5x5} respectively. We will also denote with $V_3$, $V_4$ and $V_5$ the 
	matrices whose columns are defined with the generators of the corresponding 
	subspaces described in each case in the mentioned appendices.
	$M_3=V_3\circ \overline{V_3}$, $M_4=V_4\circ \overline{V_4}$ and $M_5=V_5\circ 
	\overline{V_5}$ are also the 
	matrices defined there.
	Similarly $W_3$, $W_4$ and $W_5$ will 
	denote the matrices whose unique column is the generator of the corresponding 
	subspace $W$. In each case these supports are interior points of the 
	corresponding flag manifolds.
	
	Observe that for any $n\in\N$, $n\geq 3$, there exist $h,k,l\in\N$ such that 
	$n=3h+4k+5l$. Let us now fix a triple of those $h,k$ and $l$ and consider the  
	subspaces $V$ and $W$ defined as follows. $V$ is generated by the columns of 
	the following $n\times n$ block matrix $V_n$ formed with $h$ copies of $V_3$, 
	$k$ of $V_4$ and $l$ of $V_5$ in the diagonal 
	$$
	V_n=
	\begin{array}{c}
	\begin{array}{ccc}
	\overbrace{
		\begin{array}{c}	\hspace{.7cm} \end{array}}^{3h} 
	&
	\overbrace{
		\begin{array}{c}	\hspace{.7cm} \end{array}}^{4k} 
	&
	\overbrace{	\begin{array}{c}	\hspace{.7cm} \end{array}}^{5l} 
	\end{array}
	\\ [-2ex] 
	\begin{pmatrix}
	\begin{array}{ccc}
	\oplus_{i=1}^h V_3& 0 &0 \\
	0&	\oplus_{i=1}^k V_4 &0 \\
	0& 0 &\oplus_{i=1}^l V_5
	\end{array}   
	\end{pmatrix}
	\end{array}
	\ \text{ where } 
	\oplus_{i=1}^m V_j=
	\begin{array}{c}
	\begin{array}{c}
	\overbrace{
		\begin{array}{c}	\hspace{2cm} \end{array}}^{j\times m} 
	\end{array}
	\\ [-2ex] 
	\begin{pmatrix} 
	V_j   & 0   &\dots  & 0 \\
	0&	\ddots &\ddots    & \vdots\\
	\vdots&	\ddots &\ddots & 0 \\
	0 & \dots    &0      &V_j\end{pmatrix}
	\end{array}
	$$
	
	and $W_n$ is the $n\times 1$ matrix formed with $h$ 
	copies of $W_3$, $k$ of $W_4$ and $l$ of $W_5$ concatenated (where $W_i$, $i=1,2,3$ are the subspaces used in the appendices \ref{apendice A: ej 3x3}, \ref{apendice B: ej 4x4} and \ref{apendice C: ej 5x5}). The transpose of 
	$W_n$ is:
	$$   
	\begin{array}{lcl}
	{}&
	\begin{array}{ccc}
	\overbrace{
		\begin{array}{c}	\hspace{2cm} \end{array}}^{3h} 
	&
	\overbrace{
		\begin{array}{c}	\hspace{2cm} \end{array}}^{4k} 
	&
	\overbrace{	\begin{array}{c}	\hspace{2cm} \end{array}}^{5l} 
	\end{array}
	&{}
	\\[-2ex] 
	(W_n)^t=&    
	\begin{pmatrix}  
	W_3& \dots & W_3
	&
	W_4& \dots & W_4
	&
	W_5& \dots & W_5 
	\end{pmatrix}
	&
	\in \C^{1\times n}
	\end{array}
	$$
	%
	Now consider the subspace $V\subset \C^n$, generated by the columns of $V_n$ 
	and $W\in\C^n$ generated by $W_n$. Then it can be verified that $V$ is 
	orthogonal to $W$ and (since $\dim(V_3)=\dim(V_4)=2$ and $\dim(V_5)=3$) that 
	$\dim(V)=2h+2k+3l$ and $\dim(W)=1$.
	Moreover, considering $M_n=V_n\circ \overline{V_n}\in M_n(\R_+)$ it is easy to 
	see 
	that $\det(M_n)=\det(M_3)^h \det(M_4)^k \det(M_5)^l\neq 0$ 
	because every factor is non-zero (see  \ref{apendice A: ej 3x3}, \ref{apendice B: ej 4x4} and  \ref{apendice C: ej 5x5}). Then, the linear system
	$$
	M_n X=W_n\circ \overline{W_n}
	$$
	has a unique solution $X\in \R^{n\times 1}$.
	The concrete solution $X$ can be found considering 
	the examples of the Appendix, and satisfies $X_{i,1}>0$ for all $i=1,\dots, 
	n$.   
	Thus the pair $(V,W)\in \mathcal{F}_{(2h+2k+3l,1)}\subset\C^n=\C^{3h+4k+5l}$ 
	and is a support as in Definition \ref{def: soporte} (consider the vectors 
	$v^i=\sqrt{X_{i,1}}v_i$ in the conditions \eqref{eq: condicion soporte}, for 
	$i=1,\dots,n$, and $v_i$ the $i^\text{th}$ column of 
	$V_n$). 
	
	Now consider small perturbations $V'$ and $W'$ of the subspaces $V$ and 
	$W$ such that the dimensions of the perturbed subspaces are conserved  
	and $V'\perp W'$ holds. That is, the pair $(V', W') \in 
	\mathcal{F}{(2h+2k+3l,1)}$ and is near $(V,W)$. 
	Then, we can choose $n$ vectors of $V'$ close to the ones in the columns of 
	$V_n$  such that they generate $V'$. Similarly for $W'$. 
	Let us denote with $V_n'$ and $W_n'$ the matrices such that its columns are 
	the respective generators mentioned.
	Moreover, $V'$ and $W'$ can be chosen in a neighborhood of $V$ and $W$ in such 
	a way that the pair of matrices $V_n'$ and $W_n'$ satisfy that 
	\begin{enumerate}
		\item $\det(V_n'\circ \overline{V_n'})\neq 0$,
		\item the unique solution $X\in\C^{n\times 1}$ of $ (V_n'\circ 
		\overline{V_n'}) X=W_n'\circ 
		\overline{W_n'}$ satisfies $X_{i,1}>0$.
	\end{enumerate} 
	This implies that the pair $(V',W')$ is a support according to Definition \ref{def: soporte}.
	Then $(V,W)$ is an interior point of $\mathcal{F}_{(2h+2k+3l,1)}$.
\end{proof}

\begin{remark}
	Observe that in the decomposition used in the previous proof given by $n=3h+4k+5l$, with $h,k,l\in\N$, the term $5l$ is only needed for $n=5$. Every $n\in \N\setminus\{1,2,5\}$ can be written as $n=3h+4k$.
\end{remark} 

\begin{remark}
	Note that at the end of the proof of the previous theorem if the subspaces 
	$V'$ and $W'$ are not required to be orthogonal they still satisfy the 
	conditions (1) and (2) if they are close enough to $V$ and $W$. 
\end{remark}

 \begin{appendix}
 	Here we present examples of supports in low dimensions that are interior points of flag manifolds.
\section{Example of a support  in ${\C}^3$ that is an interior point of 
	$\mathcal{F}_{(2,1)}$}\label{apendice A: ej 3x3}	
Let us consider the dimension 2 subspace $V\subset\C^3$ generated by the 
following norm one vectors:
$$
\left\{
\begin{array}{rcl}
v^1&=&\left(\frac{1886514-7511450
	i}{\sqrt{157449642458577}},
0,-\frac{4236005-8917684
	i}{\sqrt{157449642458577}}\right),\\
v^2&=&\left(-\frac{6034458+5957865	i}{\sqrt{175782050184862}},
\frac{10006368+1934893	i}{\sqrt{175782050184862}},
0\right),\\
v^3&=&\left(-\frac{ {30683}
	+ {33081	i} }{{28}\sqrt{4664715}},\left(\frac{1537}{4}+\frac{479
	i}{4}\right)
\sqrt{\frac{3}{1554905}},\left(\frac{61}{7}+\frac{157
	i}{14}\right)
\sqrt{\frac{55}{84813}}\right),
\end{array}
\right.
$$
and the subspace $W$ generated by $ w=
\left(\frac{5- i}{2\sqrt{15}},\left(\frac{1}{2}-\frac{i}{2}\right)
\sqrt{\frac{3}{5}},\frac{2}{\sqrt{15}}\right)
$, that is orthogonal to $V$.

Then, if  
$a_1= \frac{115667}{303810}$, $a_2=\frac{85199}{222794}$ and $a_3=\frac{395794}{1670955}$
 direct computations show that $ \sum_{i=1}^3 a_i=1$ and that if 
$V_3=  
( v^1,  v^2 , v^3)$
 and 
 $M_3=V_3\circ \overline{V_3}=  
(v^1\circ \overline{v^1} , v^2\circ \overline{v^2} , v^3\circ 
 \overline{v^3})$ are the $3\times 3$ matrices with columns $v^i$ and  
 respectively $v^i\circ 
 \overline{v^i}$, for $i=1,2,3$, 
 then
 \begin{equation}\label{eq: ecuaciones ejemplo 3x3}
M_3\cdot a^t
 =\sum_{i=1}^3 a_i 
 \ \left(v^i\circ \overline{v^i}\right)=w\circ \overline{w}.
 \end{equation}
 where $a=(a_1,a_2,a_3)$, and $a^t$ is its transpose (a column matrix).  This proves that the pair $(V,W)$ is a support (consider the vectors 
 $v^i=\sqrt{a_i}\, v^i$, for $i=1,2,3$ and the 
 definition \eqref{eq: condicion soporte}). Moreover, it can be checked that the matrix $M_3$ (the one involved 
 in the equation \eqref{eq: ecuaciones ejemplo 3x3}) has non-zero 
 determinant (which proves that the numbers $a_i$, $i=1,2,3$ are unique).

 Then, it can be proved that the support $(V,W)$ is interior in the set of 
 $(2,1)$ flags $\mathcal{F}_{(2,1)}$ of $\C^3$. This follows because 
 continuous and small perturbations $w'$ and $(v')^i$ of the vectors $w$ and
 $v^i$ (with the condition $w'\perp (v')^i $, for $i=1,2,3$), produce a non-zero 
 determinant of the perturbed corresponding matrix $M_3'$ with columns $(v')^i \circ 
 \overline{(v')^i }$ for 
 $i=1,2,3$. 
 Then there are unique solutions $a_i'>0$ of the corresponding equation 
 \eqref{eq: ecuaciones ejemplo 3x3} for the 
 new vectors $(v')^i $ and $w'$. This proves that there exists a neighborhood of 
 $(V,W)$ in 
 $\mathcal{F}_{(2,1)}$ such that every pair $(V',W')$ belonging to it is a 
 support according to Definition \ref{def: soporte}.

\section{Example of a support  in $\C^4$ that is an interior point of 
	$\mathcal{F}_{(2,1)}$}\label{apendice B: ej 4x4}
Let $V\subset \C^4$ be the subspace of dimension 2 generated by the following 
norm one vectors:
$$
\left\{
\begin{array}{rcl}
v^1&=& \left(-\frac{\frac{698}{3}+75 	
i}{\sqrt{212114}},\frac{\frac{1036}{3}+51
	i}{\sqrt{212114}},\left(\frac{77}{3}-\frac{218
	i}{3}\right)
\sqrt{\frac{2}{106057}},\left(-\frac{113}{3}+\frac{49
	i}{3}\right)
\sqrt{\frac{2}{106057}}\right), \\
v^2&=&
\left(-\frac{530-\frac{655
		i}{2}}{\sqrt{1918749}},\frac{760-\frac{173
		i}{2}}{\sqrt{1918749}},\frac{\frac{219}{2}-782
	i}{\sqrt{1918749}},\frac{\frac{263}{2}+552
	i}{\sqrt{1918749}}\right), \\
v^3&=& \left(-\frac{75+\frac{45
		i}{4}}{\sqrt{29729}},\frac{
	54-\frac{365
		i}{4}}{\sqrt{29729}},-\frac
{18-\frac{243
		i}{4}}{\sqrt{29729}},-\frac
{\frac{169}{2}+\frac{159
		i}{4}}{\sqrt{29729}}\right)\\
	v^4&=&	\left(-\frac{\frac{1345}{2}+283
i}{\sqrt{1909509}},\frac{\frac{563}{2}+239
i}{\sqrt{1909509}},\frac{738+\frac{263
i}{2}}{\sqrt{1909509}},-\frac{782-\frac{519
i}{2}}{\sqrt{1909509}}\right)

	\end{array}
	\right.
$$
and $W$ the subspace generated by 
$w=
\left(\frac{\frac{1}{2}-\frac
	{i}{2}}{\sqrt{2}},\frac{\frac{1}{2}-\frac{i}{2}}{\sqrt
	{2}},\frac{\frac{1}{2}+\frac{i}{2}}{\sqrt{2}},\frac{\frac{1}{2}+\frac{i}{2}}
	 {\sqrt{2}}\right)
$.

If $a_1=\frac{20559837596768881}{124590980225106843}$, $a_2=
\frac{96813856451303497}{415303267417022810}$, $a_3=\frac{
1154873210442508}{8612279739062685}$ and  $a_4=\frac{
49954131355895969}{106792268764377294}$, it can be verified that if 
$V_4=	(v^1 , v^2  , v^3 , v^4 )$,
$M_4=	V_4\circ\overline{V_4}=(v^1\circ v^1 , v^2\circ v^2  , v^3\circ v^3 ,v^4\circ v^4 )$ and $a=(a_1,a_2,a_3,a_4)$, then 
	\begin{equation}\label{eq: ecuaciones ejemplo 4x4}
M_4\cdot a^t
	=\sum_{i=1}^4 a_i 
	\ (v^i\circ \overline{v^i})=w\circ \overline{w}.
	\end{equation}
The determinant of the matrix $M_4$ is non-zero and therefore similar 
considerations as those made in the previous example in \ref{apendice A: ej 3x3} 
can be used in order to prove that $(V,W)$ is a support of $\C^4$ that is included in an 
open subset of the flags $\mathcal{F}_{(2,1)}$.

\section{Example of a support  in $\C^5$ that is an interior point of 
	$\mathcal{F}_{(3,1)}$}\label{apendice C: ej 5x5}
Let $V\subset \C^5$ be the subspace of dimension 3 generated by the rows 
of the matrix 
$M_V =
\left(
\begin{array}{ccccc}
-\frac{19}{50}-\frac{i}{50} &
-\frac{2}{25}+\frac{19
	i}{50} &
-\frac{7}{25}+\frac{3
	i}{25} &
\frac{8}{25}+\frac{3 i}{25}
& \frac{8}{25}+\frac{11
	i}{50} \\
-\frac{1}{5}+\frac{11 i}{50}
& \frac{1}{10}+\frac{3
	i}{25} &
\frac{19}{50}-\frac{i}{5} &
\frac{19}{50}+\frac{3
	i}{10} & -\frac{21}{50} \\
\frac{29}{50} &
-\frac{1}{50}+\frac{3
	i}{10} &
-\frac{1}{10}-\frac{8
	i}{25} &
\frac{1}{5}+\frac{9 i}{50}
& \frac{1}{5}-\frac{21
	i}{50} \\
\end{array}
\right) 
$
and $W$ the subspace generated by 
$w=\left(\frac{1-i}{\sqrt{10}},\frac{1-i}{\sqrt{10}},\frac{
	1+i}{\sqrt{10}},\frac{1+i}{	\sqrt{10}},\frac{1+i}{\sqrt
	{10}}\right).
$

Now define the coefficient matrix 
$C=
\begin{pmatrix} 
-\frac{16}{5}+\frac{53 i}{10}
& -\frac{4}{5}-\frac{5
	i}{2} &
\frac{8}{5}+\frac{29 i}{10}
\\
1+\frac{i}{2} & -\frac{3}{5}
& \frac{2}{5}+\frac{13
	i}{10} \\
-\frac{13 i}{5} &
-\frac{11}{5}+\frac{29
	i}{10} & \frac{4}{5}+i \\
1+\frac{3 i}{5} &
\frac{3}{5}+\frac{7 i}{5} &
\frac{1}{2}+\frac{7 i}{10}
\\
\frac{13}{10}+\frac{3 i}{2} &
-\frac{2}{5}+\frac{9 i}{5}
& -\frac{14}{5}-\frac{3
	i}{10} \end{pmatrix}
$, 
and consider the 5 vectors belonging to $V$ obtained from the rows of the 
product 
$C\cdot M_V\in 
\C^{5\times 5}$. Let us denote those vectors (rows) with $v^1$, $v^2$, $v^3$, 
$v^4$ and $v^5$.
Then it can be checked that for
$a_1=\frac{38245034600180718292066117}{93505493283505350729949090}$, 
$a_2=\frac{876893808432404350620802}{9350549328350535072994909}$, 
$a_3=\frac{11840789324853298629489761}{93505493283505350729949090}$, 
$a_4=\frac{11483749488079211997737796}{46752746641752675364974545}$ and  
$a_5=\frac{1168323229798886630670960}{9350549328350535072994909}$ the equality 
$\sum_{i=1}^5 a_i=1$ holds and if 
$V_5=	(v^1 , v^2  , v^3 , v^4 ,v^5)$, $M_5=V_5\circ \overline{V_5}=(v^1\circ v^1 , v^2\circ v^2  , v^3\circ v^3 ,v^4\circ v^4, v^5\circ v^5  )$ and $a=(a_1,a_2,a_3,a_4,a_5)$, then 
 \begin{equation}\label{eq: ecuaciones ejemplo 5x5}
M_5\cdot a^t
=\sum_{i=1}^5 a_i 
\ (v_i\circ \overline{v_i})=w\circ \overline{w}.
\end{equation}
The determinant of the matrix $M_5$ involved in equation \eqref{eq: ecuaciones 
ejemplo 5x5} is non-zero and therefore similar considerations as those made in 
the previous examples of the Appendix can be made in order to prove that 
$(V,W)$ is a support 
that is included in an open subset of the flags $\mathcal{F}_{(3,1)}$ of $\C^5$.
\begin{remark}
	Note that the steps used to prove that the previous example $(V_5,W_5)$ is an 
	interior point of $\mathcal{F}_{(3,1)}$ in 
	$\C^5$ cannot be followed if the dimensions of the subspaces were 
	$\dim(V)=2$ and $\dim(W)=1$ as in  
	\ref{apendice A: ej 3x3} and \ref{apendice B: ej 4x4}. This is because if 
	$\dim(V)=2$ then $\rank{M_5}=\rank(V_5\circ \overline{V_5})\leq \rank(V_5) 
	\rank(\overline{V_5})=4$, and therefore $\det(M_5)=0$ in this case (for any 
	choice of $V_5$).
	This is not enough to asseverate that there is not a support 
	 in $\C^5$ that is an interior point of $\mathcal{F}_{(2,1)}$, but we have not found an example with these dimensions.
\end{remark}

\end{appendix}

\bibliographystyle{abbrv}

\end{document}